\documentclass[12pt,reqno]{amsart}
\usepackage[dvipsnames]{xcolor}
\usepackage{graphicx}
\usepackage{tikz-cd}
\usepackage{amsmath, amssymb}
\usepackage{mathrsfs}
\usepackage{amsthm}
\usepackage{multicol}
\usepackage{color}
\usepackage{bm}
\usepackage{mathtools}
\usepackage{nameref}
\usepackage{url}
\usepackage[colorlinks=true,linkcolor=blue!60!black,citecolor=teal!80!black,urlcolor=RoyalBlue]{hyperref}
\usepackage[utf8x]{inputenc} 
\usepackage[left=2.7cm,right=2.7cm,top=2.65cm,bottom=2.65cm]{geometry}
\usepackage{enumitem}
\usepackage{accents}
\usepackage{import}

\usepackage[color=red!20!,bordercolor=red, shadow]{todonotes}
\newtheorem{thm}{Theorem}[section]
\newtheorem{cor}[thm]{Corollary}
\newtheorem{lemma}[thm]{Lemma}
\newtheorem{prop}[thm]{Proposition}

\newtheorem{dfn&prop}[thm]{Definition and Proposition}
\newtheorem{dfn&thm}[thm]{Definition and Theorem}

\theoremstyle{definition}
\newtheorem{defn}[thm]{Definition}
\newtheorem{observation}[thm]{Observation}

\newtheorem*{structure}{Structure of the article and outline of the proof}

\newtheorem{discussion}[thm]{}

\theoremstyle{remark}
\newtheorem*{Claim}{Claim}

\newtheorem*{remark}{Remark}

\numberwithin{equation}{section}

\newenvironment{subproof}{\begin{proof}[Proof of claim.]}{%
	\end{proof}}

\newcommand{\Deriv}{{\rm D}}
\def\phi{\varphi}

\def\M{{\mathcal{M}}}

\def\B{{\mathcal{B}}}
\def\CB{{\mathcal{CB}}}

\def\C{{\mathbb{C}}}
\def\D{{\mathbb{D}}}
\def\N{{\mathbb{N}}}

\def\R{{\mathbb{R}}}

\def\Q{\mathbb {Q}}

\def\Or{{\mathcal{O}}}
\def\Ort{{\widetilde{\mathcal{O}}}}

\newcommand{\unbdd}[1]{\accentset{\infty}{#1}}

\newcommand{\Crit}{\operatorname{Crit}}
\newcommand{\Irr}{\operatorname{Irr}}
\newcommand{\addr}{\operatorname{addr}}
\newcommand{\Addr}{\operatorname{Addr}}
\newcommand{\ul}{\underline}

\newcommand{\ultau}{\underline{\tau}}
\newcommand{\T}{\mathcal{T}}
\newcommand{\V}{\mathcal{V}}
\newcommand{\AV}{\operatorname{AV}}
\newcommand{\CV}{\operatorname{CV}}
\newcommand{\Orb}{\operatorname{Orb}}

\def\s{{\underline s}}

\newcommand*{\defeq}{\mathrel{\vcenter{\baselineskip0.5ex \lineskiplimit0pt
			\hbox{\scriptsize.}\hbox{\scriptsize.}}}%
	=}

\newcommand{\eqdef}{=\mathrel{\vcenter{\baselineskip0.5ex \lineskiplimit0pt
			\hbox{\scriptsize.}\hbox{\scriptsize.}}}}

\title{Splitting hairs with transcendental entire functions}

\author[L. Pardo-Sim\'{o}n]{Leticia Pardo-Sim\'{o}n}

\address{Department of Mathematics \\ The University of Manchester \\ Manchester \\ M13 9PL \\ United Kingdom \\ 
	 \textsc{\newline \indent 
	   \href{https://orcid.org/0000-0003-4039-5556%
	     }{\includegraphics[width=1em,height=1em]{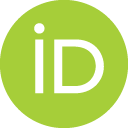} {\normalfont https://orcid.org/0000-0003-4039-5556}}
	       }}
\email{leticia.pardosimon@manchester.ac.uk}


\begin{document}

\newcommand{\new}[1]{\textcolor{magenta}{#1}}
\begin{abstract}
In recent years, there has been significant progress in the understanding of the dynamics of transcendental entire functions with bounded \textit{postsingular set}. In particular, for certain classes of such functions, a complete description of their topological dynamics in terms of a simpler model has been given inspired by methods from polynomial dynamics. In this paper, and for the first time, we give analogous results in cases when the postsingular set is unbounded. More specifically, we show that if $f$ is of finite order, has \textit{bounded criticality} on its Julia set $J(f)$, and its singular set consists of finitely many critical values that escape to infinity and satisfy a certain separation condition, then $J(f)$ is a collection of \textit{dynamic rays} or \emph{hairs}, that \emph{split} at critical points, together with their corresponding landing points. In fact, our result holds for a much larger class of functions with bounded singular set. Moreover, this result is a consequence of a significantly more general one: we provide a \textit{topological model} for the action of $f$ on its Julia set. 
\end{abstract}

\maketitle

\section{Introduction}
For a polynomial $p$ of degree $d \geq 2$, Böttcher's Theorem provides a conjugacy between $p$ and the simpler map $z \mapsto z^{d}$ in a neighbourhood of infinity. Whenever all the orbits of the critical points of $p$ are bounded (or equivalently when its Julia set $J(p)$ is connected), this conjugacy can be extended to a biholomorphic map between $\C \setminus \overline{\D}$ and the basin of infinity of $p$. In particular, it allows us to define \textit{dynamic rays} for $p$ as the curves that arise as preimages of radial rays from $\partial \D$ to $\infty$ under this conjugacy, and provide a natural foliation of the set of points of $p$ that escape to infinity under iteration. Whenever $J(p)$ is locally connected, each ray has a unique accumulation point in $J(p)$, and we say that the ray \textit{lands}. This limiting behaviour of dynamic rays has been used with great success to provide a combinatorial description of the dynamics of $p$ in $J(p)$. For example, in this situation, Douady \cite{douady_pinchedmodel} constructed a \textit{topological model} for $J(p)$ as a ``pinched disc'', that is, as the quotient of $\partial\D$ by a natural equivalence relation.

Since for a transcendental entire map, $f$, infinity is an essential singularity, Böttcher's Theorem no longer applies. Nevertheless, analogues of dynamic rays often exist, and then, it is natural to ask about their landing behaviour, and, more generally, the existence of topological models for the dynamics of $f$ on $J(f)$. Answers to these questions depend largely on its \emph{singular set} $S(f)$, that is, the closure of the set of its critical and asymptotic values, as well as on its \emph{postsingular set} $P(f) \defeq\overline{\bigcup_{n\geq 0}f^n (S(f))}$. In fact, in this paper, we restrict ourselves to the  widely studied \emph{Eremenko-Lyubich class}~$\mathcal{B}$, consisting of all transcendental entire functions with bounded singular set, \cite{eremenkoclassB}. Then, it is known that if $f$ is a finite composition of functions in $\B$ of \textit{finite order}, i.e., so that $\log \log \vert f_i(z)\vert=O(\log \vert z \vert)$ as $\vert z \vert \rightarrow \infty$, then every point in its \textit{escaping set}
\[I(f)\defeq \{z\in\C : f^n(z)\to \infty \text{ as } n\to \infty\}\]
can be connected to infinity by an escaping curve, subsequently called \textit{dynamic ray} by analogy with the polynomial case, \cite{Baranski_Trees, RRRS}. More precisely, we adopt \cite[Definition 2.2]{RRRS} and \cite[Definition~1.2]{lasse_dreadlocks}:
\begin{defn}[Dynamic rays, criniferous maps]\label{def_ray}
Let $f$ be a transcendental entire function. A \emph{ray tail} of $f$ is an injective curve $\gamma :[t_0,\infty)\rightarrow I(f)$, with $t_0>0$, such that
\begin{itemize}
\item for each $n\geq 1$, $t \mapsto f^{n}(\gamma(t))$ is injective with $\lim_{t \rightarrow \infty} f^{n}(\gamma(t))=\infty$;
\item $f^{n}(\gamma(t))\rightarrow \infty$ uniformly in $t$ as $n\rightarrow \infty$.
\end{itemize}
A \emph{dynamic ray} of $f$ is a maximal injective curve $\gamma :(0,\infty)\rightarrow I(f)$ such that the restriction $\gamma_{|[t,\infty)}$ is a ray tail for all $t > 0$. We say that $\gamma$ \emph{lands} at $z$ if $\lim_{t \rightarrow 0^+} \gamma(t)=z$, and we call $z$ the \emph{endpoint} of $\gamma$. Moreover, we say that $f$ is \emph{criniferous} if for every $z\in I(f)$, there is $N\defeq N(z)\in \N$ so that $f^n(z)$ is in a ray tail for all $n\geq N$.
\end{defn}
We remark that the accumulation set of a dynamic ray might be topologically rather complicated, and, in particular, need not be a point. Indeed, this behaviour occurs for any map in the exponential family $E_{\kappa}\colon z\mapsto e^z+\kappa$ whose singular value escapes; \cite{lasse_nonlanding}. On the contrary, for $f$ a postsingularly bounded entire function, $P(f)$ is nicely separated from infinity, where rays start. This has played a crucial role when proving that all dynamic rays of certain $f\in \B$ with bounded postsingular set land; see the seminal work of Devaney and Krych on postsingularly bounded exponentials \cite{devaney_Krych}, as well as \cite{devaney_tangerman,dierkParadox,lasseRigidity,helenaSemi,mashael}.



The dynamics of entire functions with unbounded postsingular set are far less understood. For polynomials with escaping singular values, it is still possible to extend dynamic rays when they hit critical points using Green's function in a natural way, \cite{Goldberg_Milnor,kiwi_rationalrays}. However, with the essential singularity at infinity, we encounter very different dynamics for a transcendental map $f$, and, a priori, it is not so clear what to expect in the presence of unbounded singular orbits, with the simplest case to consider being escaping orbits. For example, the presence of indecomposable continua in the closure of dynamic rays of exponential maps whose asymptotic value escapes, prevents the existence of a complete description of their Julia sets in terms of dynamic rays that land; see \cite{Devaney_Knasterlike, Devaney_Jarque_ind1} and \cite[Theorem 1.2]{lasse_nonlanding}.

In the following theorem, and for the first time, we give some answers to this question. Recall that an entire function $f$ has \textit{bounded criticality} on its Julia set if $J(f)$ contains no finite asymptotic value of $f$, and the local degree of $f$ at the points in $J(f)$ is uniformly bounded.
\begin{thm}[Landing of rays for functions with escaping singular orbits] \label{thm_1intro} Let $f$ be a finite composition of class $\B$ functions of finite order. Suppose that $S(f)$ is a finite collection of critical values that escape to infinity, $f$ has bounded criticality on $J(f)$, and there exists $\epsilon>0$ so that $\vert w-z\vert \geq \epsilon\max\{\vert z \vert, \vert w \vert\}$ for all distinct $z,w \in P(f)$. Then, every dynamic ray of $f$ lands, and every point in $J(f)$ is either on a dynamic ray or it is the landing point of at least one such ray.
\end{thm}

The hypotheses in Theorem~\ref{thm_1intro} will be discussed later, but first we note that for any $f$ satisfying them, $J(f)=\C$. Moreover, since $I(f)$ contains critical values, dynamic rays \textit{split} at critical points. This can be illustrated with $f(z)=\cosh(z)$. In this case, $S(f)=\CV(f)=\{-1, 1\}$, and $P(f)$ equals $S(f)$ together with the orbit of $f(-1)=f(1)$, which consists of a sequence of positive real points converging to infinity at an exponential~rate. 
\begin{figure}[htb]
	\centering
	\resizebox{0.6\textwidth}{!}{\begingroup%
  \makeatletter%
  \providecommand\color[2][]{%
    \errmessage{(Inkscape) Color is used for the text in Inkscape, but the package 'color.sty' is not loaded}%
    \renewcommand\color[2][]{}%
  }%
  \providecommand\transparent[1]{%
    \errmessage{(Inkscape) Transparency is used (non-zero) for the text in Inkscape, but the package 'transparent.sty' is not loaded}%
    \renewcommand\transparent[1]{}%
  }%
  \providecommand\rotatebox[2]{#2}%
  \ifx\svgwidth\undefined%
    \setlength{\unitlength}{368.50393066bp}%
    \ifx\svgscale\undefined%
      \relax%
    \else%
      \setlength{\unitlength}{\unitlength * \real{\svgscale}}%
    \fi%
  \else%
    \setlength{\unitlength}{\svgwidth}%
  \fi%
  \global\let\svgwidth\undefined%
  \global\let\svgscale\undefined%
  \makeatother%
  \begin{picture}(1,0.69230772)%
    \put(0,0){\includegraphics[width=\unitlength]{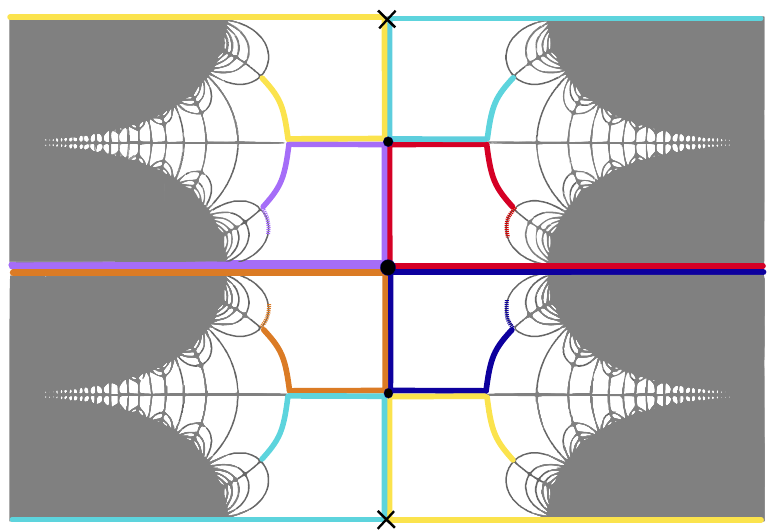}}%
    \put(0.51311181,0.21003289){\color[rgb]{0,0,0}\makebox(0,0)[lb]{\smash{$\fontsize{9pt}{1em} -i\frac{\pi}{2}$}}}%
    \put(0.51591789,0.35817888){\color[rgb]{0,0,0}\makebox(0,0)[lb]{\smash{$0$}}}%
    \put(0.51520385,0.52821125){\color[rgb]{0,0,0}\makebox(0,0)[lb]{\smash{$\fontsize{9pt}{1em} i\frac{\pi}{2}$}}}%
  \end{picture}%
\endgroup}
	\caption[Caption for LOF]{In colour, some ray tails of $f(z)=\cosh(z)$. Both the red and dark-blue ones are mapped into themselves, and the rest of coloured tails are first and second preimages of these tails. Further iterated preimages, some depicted in grey, lead to further extensions of these ray tails.\protect\footnotemark}
	\label{fig:rays_cosh}
\end{figure}
\footnotetext{Original picture by L. Rempe, modified for this paper. It first appeared in \cite[p.~163]{mio_thesis}.}

\noindent By \cite[Theorem 6.4]{dierkCosine}, $f$ is criniferous. Note that $0$ is a critical point, and it is easy to check that $(-\infty, 0]$ and $[0, \infty)$ are both ray tails. The vertical segments $[0, -i\pi/2]$ and $[0, i\pi/2]$ are mapped univalently to $[0,1]\subset \R^+$, and thus, the union of each segment with either one of the ray tails $(-\infty, 0]$ and $[0, \infty)$, forms a different ray tail. We can think of this structure as four ray tails that partially overlap pairwise. 
Their endpoints $-i\pi/2$ and $i\pi/2$ are preimages of $0$, and so the structure described has a preimage attached to each of them, see Figure \ref{fig:rays_cosh}. This leads again to two possible extensions of each ray tail. We show in \cite{mio_signed_addr} that for criniferous maps with escaping critical values, such extensions can be made in a systematic and dynamically meaningful way that leads to a foliation of their escaping sets into piecewise-overlapping rays. The additional assumptions in Theorem~\ref{thm_1intro} guarantee that such rays always land; see \cite{mio_cosine} for more details on the dynamics of~$\cosh$.

Theorem~\ref{thm_1intro} is a consequence of a more general result: in analogy to Douady's ``Pinched Disc Model'' for polynomials, we construct a topological model for the dynamics on the Julia set of any $f$ satisfying its hypotheses. So far, all existing models for transcendental entire maps regarded functions in $\B$ with bounded postsingular set. The seminal work in this direction is \cite{AartsOversteegen}, where it is shown that the Julia set of certain exponential and sine maps is homeomorphic to a topological object known as \textit{straight brush} (see Definition \ref{def_brush}). When this occurs, the Julia set is said to be a \textit{Cantor bouquet}; compare with \cite{Bhattacharjee_exp, lasseTopExp} for other parameters in the exponential family. Subsequently, it was shown in \cite{lasseBrushing} that if $f$ is a finite composition of finite order maps in $\B$ and of \textit{disjoint type}, i.e., with connected Fatou set $F(f)$ and such that $P(f) \Subset F(f)$, then $J(f)$ is a Cantor bouquet.

Understanding the dynamics of disjoint type functions is particularly useful, since if $f\in \B$, then for $\lambda \in \C\setminus \{0\}$ with $\vert \lambda \vert$ small enough, the function $\lambda f$ is of disjoint type. In particular, $\lambda f$ is in the \textit{parameter space} of $f$, that is, $f$ and $\lambda f$ are quasiconformally equivalent. This implies that their dynamics near infinity are related by a certain analogue of Böttcher's Theorem for transcendental maps; \cite{lasseRigidity}. One might regard disjoint type functions as having the simplest dynamics among those in the parameter space of $f$, and, in particular, they play an analogous role for $f$ as $z\mapsto z^d$ does for a polynomial of degree~$d$. If, in addition, $f$ is of finite order, so is any disjoint type map $g\defeq \lambda f$, and by the result in \cite{lasseBrushing} mentioned earlier, $J(g)$ is a Cantor bouquet. This is used in \cite[Theorem~5.2]{lasseRigidity} to show that when $f$ is hyperbolic, $g\vert_{I(g)}$ and $f\vert_{I(f)}$ are conjugate, and that $f$ and $g$ are semiconjugate on their Julia sets, a result extended in \cite{helenaSemi} to \textit{strongly subhyperbolic} maps. See also \cite{mashael} for a further generalization.

The function $f(z) = \cosh(z)$ illustrates the fact that functions satisfying the hypotheses of Theorem~\ref{thm_1intro} may have critical values in their dynamic rays. It follows that any model that describes their dynamics must reflect the splitting of ray tails at critical points. In particular, since, roughly speaking, a Cantor bouquet is a collection of disjoint curves, called \textit{hairs}, $J(f)$ can no longer be a Cantor bouquet, and so for $g\defeq \lambda f$ of disjoint type, $g\vert_{I(g)}$ and $f\vert_{I(f)}$ can no longer be conjugate. However, the analysis performed on the map $\cosh$, where each ray tail ``splits into two'' at critical points, suggests considering two copies of each hair of $J(g)$, and mapping each copy to one of the two possible extensions. With that aim, we build a topological model for $f\vert_{J(f)}$ the following way: we consider the set $J(g)_\pm \defeq J(g)\times \{-,+\}$, that is endowed with a topology that preserves the order of rays at infinity, see \S\ref{sec_model}, and call it a \textit{model space} for $f$. Then, we define an \textit{associated model function} $\tilde{g} \colon J(g)_\pm \rightarrow J(g)_\pm$ to act as $g$ on the first coordinate, and as the identity on the second. We let $I(g)_{\pm}\defeq I(g)\times \{ -,+\}$ and show the following:
\begin{thm}[Semiconjugacy to model space]\label{thm_main_intro1}Let $f$ be as in Theorem~\ref{thm_1intro}, let $J(g)_\pm$ be a model space for $f$, and let $\tilde{g}$ be its associated model function. Then, there exists a continuous surjective function $\phi:J(g)_{\pm} \rightarrow J(f)$ so that $f\circ\phi = \phi\circ \tilde{g}$. Moreover, $\phi(I(g)_{\pm})=I(f)$.
\end{thm}

In fact, we are able to prove significantly stronger results than those of Theorems \ref{thm_1intro} and \ref{thm_main_intro1}. However, they require some rather technical definitions, which is why we started with the more straightforward results of Theorems \ref{thm_1intro} and \ref{thm_main_intro1}. Firstly, we note that the conditions we imposed on $P(f)$ ensure that $f$ expands an \textit{orbifold metric} that sits on a neighbourhood of $J(f)$. We prove this result in \cite{mio_orbifolds} for a much larger class of maps. More precisely, we say that $f\in \B$ is \textit{strongly postcritically separated} if $P(f)\cap F(f)$ is compact, $f$ has bounded criticality on $J(f)$, there is a uniform bound on the number of critical points in the orbit of any $z\in J(f)$, and there is $\epsilon>0$ so that for any distinct $z,w\in P_J$, $\vert z-w\vert\geq \epsilon \max\{\vert z \vert, \vert w \vert\}$. In particular, in addition to escaping critical values, strongly postcritically separated maps might contain preperiodic postsingular points in their Julia set; see \S \ref{sec_postsep}. 

Secondly, as mentioned earlier, the assumption on $f$ being a finite composition of functions of finite order in $\B$ guarantees that $f$ is criniferous, and implies that $J(\lambda f)$ is a Cantor bouquet whenever $\lambda f$ is of disjoint type. In \cite{mio_thesis}, we introduce the more general \textit{class~$\CB$}, consisting of all $f \in \B$ so that $J(\lambda f)$ is a Cantor bouquet for $\vert \lambda \vert$ sufficiently small. As an extension of \cite[Theorem 1.2]{RRRS}, it is shown in \cite{mio_newCB} that all maps in $\CB$ are criniferous, and so this is a natural class to consider when studying functions with dynamic rays. Since for $f\in \CB$, any disjoint type $\lambda f$ has a Cantor bouquet Julia set, we can define a topological model for $f\vert_{J(f)}$ as we did in Theorem \ref{thm_main_intro1}.

We are now able to state the main result of this paper, from which our other results follow quickly. This concerns all functions in $\CB$ that are also strongly postcritically separated.
\begin{thm}[Semiconjugacy to model space]\label{thm_main_intro}Let $f \in \mathcal{CB}$ be strongly postcritically separated, let $J(g)_\pm$ be a model space for $f$ and let $\tilde{g}$ be its associated model function. Then, there exists a continuous surjective function
	\begin{equation}\label{eq_phi_intro}
	\phi:J(g)_{\pm} \rightarrow J(f)
	\end{equation}
	so that $f\circ\phi = \phi\circ \tilde{g}$. Moreover, $\phi(I(g)_{\pm})=I(f)$.
\end{thm}

We remark that the the semiconjugacy in Theorem \ref{thm_main_intro} respects the natural conjugacy at infinity from \cite{lasseRigidity}. Since all functions considered in Theorem~\ref{thm_main_intro1} are in $\mathcal{CB}$ and strongly postcritically separated, Theorem~\ref{thm_main_intro1} follows from Theorem~\ref{thm_main_intro}. Likewise, Theorem~\ref{thm_1intro} is a consequence of the following corollary:
\begin{cor}[Landing of rays for strongly postcritically separated functions in $\CB$] \label{cor_intro} Under the assumptions of Theorem~\ref{thm_main_intro}, every dynamic ray of $f$ lands, and every point in $J(f)$ is either on a dynamic ray, or is the landing point of at least one such ray.
\end{cor}

\begin{structure}
The proof of Theorem \ref{thm_main_intro} combines several different tools, some of them developed in separate papers. For the reader's convenience, all results required in this paper are included in sections \ref{sec_symbolic}-\ref{sec_postsep}, and as a roadmap, we now highlight the main ideas.

\noindent Let $f\in \CB$ as in Theorem \ref{thm_main_intro}, and let $g\defeq \lambda f$ of disjoint type for some $\lambda \in \C \setminus \{0\}.$ 

\begin{itemize}[leftmargin=\dimexpr\labelwidth + 4\labelsep\relax]
\item[\S2.] Since $f\in \CB$, $f$ is criniferous. Moreover, by assumption, $J(f)$ has no asymptotic values on its Julia set. In \cite{mio_signed_addr}, combinatorics for these maps are developed, where the extensions of rays at critical points, as described for the map $\cosh$, are formalized. We define the set of \textit{signed addresses} $\Addr(f)_\pm\defeq \Addr(f)\times \{-,+\}$, that we endow with a topology such that each $z \in I(f)$ has at least two signed
addresses that depend continuously on $z$. Then, we have that $I(f)$ is a collection of so-called \textit{canonical rays}
\begin{equation}\label{eq_outline2}
\{\Gamma(\s,\ast)\}_{(\s,\ast)\in \Addr(f)_\pm},
\end{equation}
that overlap piecewise between (preimages of) critical points. Moreover, we can write each canonical ray as a collection of nested curves, $\Gamma(\s,\ast)=\bigcup_{n\geq 0}\gamma^n_{(\s, \ast)}$, such that we have well-defined $n$-th inverse branches $f^{-n}_{(\s, \ast)}$ in neighbourhoods of these curves.

\item[\S3.] Since $f\in \B$, \cite[\S 3]{lasseRigidity} provides us with a homeomorphism $\theta$ that conjugates $f$ and $g$ in subsets of the Julia sets. In \cite{mio_newCB}, the fact that for $f\in \CB$, $J(g)$ is a Cantor bouquet is exploited to extend continuously 
\begin{equation} \label{eq_outline1}
\theta\colon J(g)\to J(f),
\end{equation}
ensuring that points are \textit{moved} in a controlled way.

\item[\S4.] Our arguments will rely on expansion of $f$. However, with escaping singular values, $f$ is not a covering map in a neighbourhood of $J(f)$. Instead, it is shown in \cite{mio_orbifolds} that $f$ is an \textit{orbifold covering map} that expands an \textit{orbifold metric} that sits in a neighbourhood of $J(f)$. In addition, we will use a modified notion of homotopy that will allow us to pull back curves of uniformly bounded orbifold length.
\item[\S5.] We define the model space $J(g)_\pm$ and associated model function $\tilde{g}\colon J(g)_\pm \to J(g)_\pm$, as sketched before. We endow $J(g)_\pm$ with a topology that is related to that of $\Addr(f)_\pm$, and prove properties of this space.
\vspace{-17pt}
\begin{figure}[htb]
\begingroup%
  \makeatletter%
  \providecommand\color[2][]{%
    \errmessage{(Inkscape) Color is used for the text in Inkscape, but the package 'color.sty' is not loaded}%
    \renewcommand\color[2][]{}%
  }%
  \providecommand\transparent[1]{%
    \errmessage{(Inkscape) Transparency is used (non-zero) for the text in Inkscape, but the package 'transparent.sty' is not loaded}%
    \renewcommand\transparent[1]{}%
  }%
  \providecommand\rotatebox[2]{#2}%
  \newcommand*\fsize{\dimexpr\f@size pt\relax}%
  \newcommand*\lineheight[1]{\fontsize{\fsize}{#1\fsize}\selectfont}%
  \ifx\svgwidth\undefined%
    \setlength{\unitlength}{388.34645669bp}%
    \ifx\svgscale\undefined%
      \relax%
    \else%
      \setlength{\unitlength}{\unitlength * \real{\svgscale}}%
    \fi%
  \else%
    \setlength{\unitlength}{\svgwidth}%
  \fi%
  \global\let\svgwidth\undefined%
  \global\let\svgscale\undefined%
  \makeatother%
  \begin{picture}(1,0.49635036)%
    \lineheight{1}%
    \setlength\tabcolsep{0pt}%
    \put(0.10960808,0.39996006){\color[rgb]{0,0,0}\makebox(0,0)[lt]{\lineheight{1.25}\smash{\begin{tabular}[t]{l}\fontsize{11pt}{1em}$(z,+)$\end{tabular}}}}%
    \put(0.48883019,0.39882606){\color[rgb]{0,0,0}\makebox(0,0)[lt]{\lineheight{1.25}\smash{\begin{tabular}[t]{l}$\tilde{g}^n$\end{tabular}}}}%
    \put(0.0052581,0.44870332){\color[rgb]{0.50196078,0.50196078,0.50196078}\makebox(0,0)[lt]{\lineheight{1.25}\smash{\begin{tabular}[t]{l}\fontsize{9pt}{1em} $J(g)_{\pm}$\end{tabular}}}}%
    \put(0.78500333,0.23574965){\color[rgb]{0,0,0}\makebox(0,0)[lt]{\lineheight{1.25}\smash{\begin{tabular}[t]{l}$\theta$\end{tabular}}}}%
    \put(0.11379946,0.3365583){\color[rgb]{0,0,0}\makebox(0,0)[lt]{\lineheight{1.25}\smash{\begin{tabular}[t]{l}\fontsize{11pt}{1em}$(z,-)$\end{tabular}}}}%
    \put(0,0){\includegraphics[width=\unitlength,page=1]{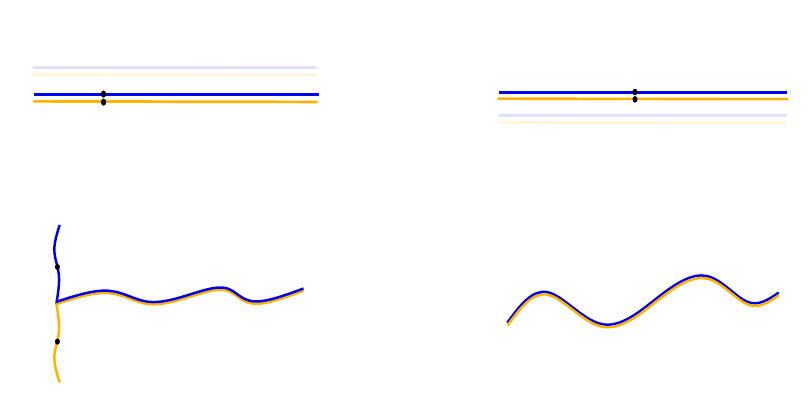}}%
    \put(0.76404327,0.40630703){\color[rgb]{0,0,0}\makebox(0,0)[lt]{\lineheight{1.25}\smash{\begin{tabular}[t]{l}\fontsize{11pt}{1em}$\tilde{g}^n(z,+)$\end{tabular}}}}%
    \put(0.76495984,0.33465415){\color[rgb]{0,0,0}\makebox(0,0)[lt]{\lineheight{1.25}\smash{\begin{tabular}[t]{l}\fontsize{11pt}{1em}$\tilde{g}^n(z,-)$\end{tabular}}}}%
    \put(0.76602772,0.06751358){\color[rgb]{0,0,0}\makebox(0,0)[lt]{\lineheight{1.25}\smash{\begin{tabular}[t]{l}\fontsize{11pt}{1em}$ \theta(g^n(z))$\end{tabular}}}}%
    \put(0.07956505,0.16603863){\color[rgb]{0,0,0}\makebox(0,0)[lt]{\lineheight{1.25}\smash{\begin{tabular}[t]{l}\fontsize{11pt}{1em}$ \phi_n(z,+)$\end{tabular}}}}%
    \put(0.08300031,0.06683009){\color[rgb]{0,0,0}\makebox(0,0)[lt]{\lineheight{1.25}\smash{\begin{tabular}[t]{l}\fontsize{11pt}{1em}$ \phi_n(z,-)$\end{tabular}}}}%
    \put(0,0){\includegraphics[width=\unitlength,page=2]{outline.pdf}}%
    \put(0.49272817,0.1698493){\color[rgb]{0,0,0}\makebox(0,0)[lt]{\lineheight{1.25}\smash{\begin{tabular}[t]{l}$f^{-n}_{(\underline{s},+)}$\end{tabular}}}}%
    \put(0.49107781,0.05700567){\color[rgb]{0,0,0}\makebox(0,0)[lt]{\lineheight{1.25}\smash{\begin{tabular}[t]{l}$f^{-n}_{(\underline{s},-)}$\end{tabular}}}}%
    \put(0,0){\includegraphics[width=\unitlength,page=3]{outline.pdf}}%
    \put(0.01233221,0.23773018){\color[rgb]{0.50196078,0.50196078,0.50196078}\makebox(0,0)[lt]{\lineheight{1.25}\smash{\begin{tabular}[t]{l}\fontsize{9pt}{1em}$ J(f)$\end{tabular}}}}%
    \put(0.30229924,0.09351058){\color[rgb]{0.50196078,0.50196078,0.50196078}\makebox(0,0)[lt]{\lineheight{1.25}\smash{\begin{tabular}[t]{l}\fontsize{9pt}{1em}$ \gamma^n_{(\underline{s}, -)}$\end{tabular}}}}%
    \put(0.29107912,0.14836634){\color[rgb]{0.50196078,0.50196078,0.50196078}\makebox(0,0)[lt]{\lineheight{1.25}\smash{\begin{tabular}[t]{l}\fontsize{9pt}{1em} $\gamma^n_{(\underline{s}, +)}$\end{tabular}}}}%
    \put(0.90822187,0.15198379){\color[rgb]{0.50196078,0.50196078,0.50196078}\makebox(0,0)[lt]{\lineheight{1.25}\smash{\begin{tabular}[t]{l}\fontsize{9pt}{1em} $\gamma^0_{(\sigma^n(\underline{s}), \pm)}$\end{tabular}}}}%
  \end{picture}%
\endgroup%
\caption{A schematic of the functions involved in the proof of Theorem \ref{thm_main_intro}.}
	\label{figure_outline}
\end{figure}
\vspace{-5pt}
\item[\S6.] We will obtain the map $\phi$ that semiconjugates $\tilde{g}\vert_{J(g)_\pm}$ to $f\vert_{J(f)}$ as the limit of a sequence of maps $\{\phi_n\}_{n\geq 0}$, that, in a rough sense, are defined as follows: Let $(z, \ast)\in J(g)_\pm$, where $z\in J(g)$ and $\ast \in \{-,+\}$. We iterate forwards $n$ times in the model space to obtain $\tilde{g}^n(z, \ast)$. Then, we \textit{project} to the dynamical plane of $f$ by using the map $\theta$ in the first coordinate, that is, we get $\theta(g^n(z))$. Depending on the sign ``$\ast$'' of our original point and on its signed address $(\s, \ast)$, we pull back under an appropriate inverse branch $f^{-n}_{(\s, \ast)}$, to obtain $\phi_n(z, \ast)$. In particular, this point belongs to $\gamma^n_{(\s, \ast)}$, see Figure~\ref{figure_outline}. Using orbifold expansion, we show that $\{\phi_n(z,\ast)\}_{n\geq 0}$ forms a Cauchy sequence, and so converges to a limit function $\phi$. All maps involved are carefully defined to be continuous, and so continuity of $\phi$ will follow. Finally, $\phi$ maps each component of $J(g)_\pm$ to the closure of a canonical ray from \eqref{eq_outline2}, and so we obtain a description of $J(f)$ as a collection of canonical rays that land.
\end{itemize}
\end{structure}

\subsection*{Basic notation}
As introduced throughout this section, the Fatou, Julia and escaping set of an entire function $f$ are denoted by $F(f)$, $J(f)$ and $I(f)$ respectively. The set of critical values is $\CV(f)$, that of asymptotic values is $\AV(f)$, and the set of critical points will be $\Crit(f)$. The set of singular values of $f$ is $S(f)$, and $P(f)$ denotes the postsingular set. Moreover, $P_{J}\defeq P(f)\cap J(f)$. We denote the complex plane by $\C$, the Riemann sphere by $\widehat{\C}$, and the disc centred at zero with radius $R$ by $\D_R$. We will indicate the closure of a domain $U$ by $\overline{U}$, which must be understood to be taken in $\C$. $A\Subset B$ means that $A$ is compactly contained in $B$. The annulus with radii $a,b \in \R^+$, $a<b$, will be denoted by $A(a,b)\defeq\lbrace w\in \C : a< \vert w \vert < b \rbrace$. For a holomorphic function $f$ and a set $A$, $\Orb^{-}(A)\defeq \bigcup^{\infty}_{n=0} f^{-n}(A)$ and $\Orb^{+}(A)\defeq \bigcup^{\infty}_{n=0} f^{n}(A)$ are the respective backward and forward orbit of $A$ under $f$.
\subsection*{Acknowledgements}
I am very grateful to my supervisors Lasse Rempe and Dave Sixsmith for their continuous help and advice. I also thank Vasiliki Evdoridou, Daniel Meyer, Phil Rippon and the referees for valuable comments.

\section{Combinatorics: signed addresses and inverse branches }\label{sec_symbolic}
This section summarizes the combinatorial concepts and results developed in \cite{mio_signed_addr} on criniferous maps with escaping critical values that we shall require. 

We start by recalling the widely used notion of \textit{external address} for functions in $\B$, that allows to assign symbolic dynamics to points whose orbit stays away from a neighbourhood of their singular set.
\begin{defn}[Tracts, fundamental domains]\label{def_fund}
Fix $f\in\B$ and let $D$ be a bounded Jordan domain around the origin, containing $S(f)$ and $f(0)$. Each connected component of $f^{-1}(\C\setminus \overline{D})$ is a \emph{tract} of $f$. Let $\delta$ be an arc connecting a point of $\overline{D}$ to infinity in the complement of the closure of the tracts. Denote
\begin{equation}\label{eq_deffund}
\mathcal{W} \defeq \C\setminus( \overline{D}\cup\delta).
\end{equation}
Each connected component of $f^{-1}(\mathcal{W})$ is a \emph{fundamental domain} of $f$, and we call the collection of all of them an \textit{alphabet of fundamental domains}, that we denote by $\mathcal{A}(D,\delta)$. Moreover, for each $F\in \mathcal{A}(D,\delta)$, $\unbdd{F}$ is the unbounded connected component of $F\setminus \overline{D}$. 
\end{defn}

\begin{defn}[External addresses] \label{def_extaddr}
Let $f\in\B$ and let $\mathcal{A}(D,\delta)$ be an alphabet of fundamental domains. An \emph{(infinite) external address} is a sequence $\s= F_0 F_1 F_2 \dots$ of elements in $\mathcal{A}(D,\delta)$. Moreover, for each external address $\s$, we let
\begin{equation}\label{eq_Js}
J_{\s}\defeq\left\{z\in\C\colon f^n(z)\in\unbdd{F_n} \text{ for all $n\geq 0$}\right\},
\end{equation}
and denote by $\Addr(f)$ the set of all $\s$ for which $J_{\s}$ is non-empty. If $z\in J_\s$ for some $\s \in \Addr(f)$, then we say that $z$ \textit{has (external) address} $\s.$ Moreover, $\sigma$ stands for the one-sided \emph{shift operator} on external addresses; that is, $\sigma(F_0 F_1 F_2\ldots ) = F_1F_2 \ldots$. 
\end{defn}

\begin{observation}[Points with external address]\label{obs_addrz} Following the definition above, the sets in \eqref{eq_Js} lie entirely in $J(f)$, see \cite[Lemma 2.6]{lasse_dreadlocks}. Since these sets are by definition pairwise disjoint, whenever it is defined, the external address of a point is unique. Moreover, it follows from \cite[Proposition 2.8]{helenaSemi} that in the special case when $f$ is of disjoint type, all points in $J(f)$ have a (unique) external address. That is,
	\begin{equation}\label{eq_continuadisj}
	f \text { is of disjoint type }\quad \Rightarrow \quad J(f)=\bigcup_{\s\in \Addr(f)}J_\s,
	\end{equation}
where the union is disjoint.
\end{observation}
However, not all escaping points of every function in $\B$ have an external address, as in particular occurs for $f\in\B$ with singular values in $I(f)$. This problem is resolved in \cite{mio_signed_addr} for criniferous maps in $\B$ with escaping critical values. More precisely, recall from the introduction that whenever a ray tail contains a critical value, there are components of its preimage with critical points, that can be interpreted as tails that \textit{split} or \textit{break} at them, and can be extended to overlap pairwise. In order to assign dynamically meaningful combinatorics to these points, we introduce the concept of \textit{signed address}:

\begin{discussion}[Space of signed addresses]\label{discussion_signed_addr}
Let $f\in \B$, and let $\Addr(f)$ be a set of external addresses defined from an alphabet of fundamental domains $\mathcal{A}(D,\delta)$. Consider the set
$$\Addr(f)_\pm \defeq \Addr(f) \times \{-,+\},$$
that we shall endow with a topology, that we define as follows. There is a natural \emph{cyclic order} on the alphabet $\mathcal{A}(D,\delta)$ together with the curve $\delta$: if $X, Y, Z\in \mathcal{A}(D,\delta) \cup \{\delta\}$, then we write
\begin{equation}\label{eq_orderinfty}
[X,Y,Z]_{\infty} \: \: \Leftrightarrow \: \: Y \text{ tends to infinity between }X\text{ and } Z \text{ in positive orientation.}
\end{equation}
From this cyclic order, it is possible to define a \textit{lexicographical order} on the set $\Addr(f)$: We can define a linear order on the set of fundamental domains by ``cutting'' $\delta$ the following way: $$F < \tilde{F} \quad \text{ if and only if } \quad [\delta, F, \tilde{F}]_\infty.$$ 
Then, the set of fundamental domains becomes totally ordered, and this order gives rise to a lexicographical order ``$<_{_\ell}$'' on external addresses, defined in the usual sense.

Let us give the set $\lbrace -,+\rbrace$ the order $\lbrace -\rbrace \prec\lbrace +\rbrace$. Define the linear order on $\Addr(f)_\pm$
\begin{equation} \label{eq_linear_addr}
(\ul{s}, \ast )<_{_A} (\ul{\tau}, \star) \qquad \text{ if and only if } \qquad \ul{s} <_{_\ell} \ul{\tau} \quad \text{ or } \quad \ul{s} =_{_\ell} \ul{\tau}\: \text{ and } \: \ast \prec \star,
\end{equation}
where the symbols ``$\ast, \star$'' denote generic elements of $\lbrace-, +\rbrace.$ 
This linear order gives rise to a cyclic order: for $a,x,b \in \Addr(f)_\pm$,
\begin{equation*} \label{eq_orderAfpm}
[a,x,b]_{_A} \quad \text{if and only if} \quad a<_{_A} x<_{_A} b \quad \text{ or } \quad x <_{_A}b <_{_A} a \quad \text{ or }\quad b <_{_A} a <_{_A} x.
\end{equation*}
This cyclic order allows us to provide the set $\Addr(f)_\pm$ with a topology $\tau_A$: given two different elements $(\s,\ast),(\ultau,\star) \in \Addr(f)_{\pm}$, we define the \textit{open interval} from $(\s,\ast)$ to $(\ultau,\star)$, denoted by $((\s,\ast),(\ultau,\star))$, as the set of all signed addresses $(\ul{\alpha}, \cdot)\in \Addr(f)_\pm$ such that $[(\s,\ast), (\ul{\alpha}, \cdot), (\ultau,\star)]_{A}$. The collection of all these open intervals forms a base for the \textit{cyclic order topology.} 
\end{discussion}
\begin{defn}[Signed external addresses for criniferous functions]\label{defn_signedaddr} Let $f\in \B$ be a criniferous function and let $(\Addr(f)_\pm, \tau_A)$ as in \ref{discussion_signed_addr}. A \emph{signed (external) address} for $f$ is any element of $\Addr(f)_\pm$. 
\end{defn}
Our next goal is to define signed addresses for all escaping points of certain criniferous functions. Since for criniferous maps in $\B$ the sets in \eqref{eq_Js} contain unbounded curves, see \cite[Theorem 2.12]{mio_signed_addr}, we will achieve our goal by extending them in a careful and systematic way. The following definition specifies which sub-curves can be used to perform such extensions:
\begin{defn}(Initial configuration of tails)\label{def_initial_conf} Let $f\in \B$ so that for each $\s \in \Addr(f)$, there exists a curve $\gamma^0_\s\subset J_\s$ that is either a ray tail, or a dynamic ray possibly with its endpoint. The set of curves $\{\gamma^0_\s\}_{\s \in \Addr(f)}$ is a \emph{valid initial configuration} for $f$ if, for each $\s\in \Addr(f)$, $f(\gamma^0_\s) \subset \gamma^0_{\sigma(\s)}$ and
\begin{equation}\label{eq_S0}
I(f) \subset \bigcup_{\s \in \Addr(f)} \bigcup_{n\geq 0} f^{-n}(\gamma^0_\s).
\end{equation}	
\end{defn}	 

The next theorem, which gathers \cite[Definition 3.5, Theorem 3.8 and Observation~3.13]{mio_signed_addr}, tells us that the escaping sets of the functions we consider, can be described as a collection of rays, that we call \textit{canonical}, indexed by signed addresses. For a more detailed description on the construction and overlapping of these curves, we refer to \cite[\S 3]{mio_signed_addr}, or \cite[\S5]{mio_cosine} for the maps $\cosh$ and $\cosh^2$.

\begin{dfn&thm}[Canonical rays]\label{thm_signed} \normalfont Let $f\in \B$ be criniferous such that $J(f)\cap \AV(f)=\emptyset$. Let $\{\gamma^0_\s\}_{\s \in \Addr(f)}$ be a valid initial configuration for $f$. Then, for each $(\s, \ast) \in \Addr(f)_\pm$, there exists a curve $\Gamma(\s, \ast)$, that is either a ray tail or a dynamic ray possibly with its endpoint. Moreover, $\Gamma(\s, \ast)$ can be written as a nested union 
\begin{equation}\label{eq_canonical}
\Gamma(\s, \ast)= \bigcup_{n\geq 0}\gamma^n_{(\s, \ast)},
\end{equation}
satisfying:
\begin{enumerate}[label=(\alph*)]
\item $\gamma^0_{(\s,-)}\defeq\gamma^0_{(\s,+)}\defeq\gamma^0_\s\subset J_\s$; 
\item for all $n\geq 1$, $\gamma^{n-1}_{(\s, \ast)} \subseteq \gamma^{n}_{(\s, \ast)}$ and $f \colon \gamma^{n}_{(\s, \ast)} \to \gamma^{n-1}_{(\sigma(\s), \ast)}$ is a bijection.
\label{item:tails_bijection}
\end{enumerate}
We say that $\Gamma(\s, \ast)$ is a \emph{canonical ray}. Landing of all canonical rays implies landing of all dynamic rays in $J(f)$, and 
\begin{equation}\label{eq_thmcanonical}
I(f) \subset \bigcup_{(\s, \ast)\in \Addr(f)_\pm}\Gamma(\s, \ast).
\end{equation}
\end{dfn&thm}
\begin{defn}[Signed addresses for escaping points]\label{def_addrpm} Following Theorem~\ref{thm_signed}, for each $z\in I(f)$, we say that $z$ has \emph{signed (external) address} $(\s, \ast)$ if $ z\in \Gamma(\s, \ast)$, and we denote by $\Addr(z)_\pm$ the set of all signed addresses of $z$. By \cite[Proposition 3.9 and Observation~3.11]{mio_signed_addr}, for each $z\in I(f)$,
\begin{equation}\label{eq_singed_adddr}
\# \Addr(z)_\pm= 2 \prod^{\infty}_{j=0}\deg(f,f^{j}(z))<\infty.
\end{equation}
\end{defn}

One of the key features of canonical rays is that, given the topology we have provided $\Addr(f)_\pm$ with, we can define inverse branches on neighbourhoods of the nested curves in \eqref{eq_canonical} in such a way that the branches agree in whole intervals of addresses. These branches will be instrumental in the proof of Theorem~\ref{thm_main_intro}. The following theorem is \cite[Proposition~4.5 and Theorem 4.6]{mio_signed_addr}.
\begin{thm}[Inverse branches for canonical tails]\label{thm_inverse_hand}\normalfont Let $f\in \B$ be criniferous such that $J(f)\cap \AV(f)=\emptyset$, $P(f)\setminus I(f)$ is bounded and $S(f)\cap I(f)$ is finite. Then, there is a choice of $\Addr(f)$ such that, if we apply Theorem~\ref{thm_signed} to $f$ given any valid initial configuration, then, for each $n\in \N$ and $(\s, \ast) \in \Addr(f)_\pm$, there exists a, not necessarily open, neighbourhood
$$\tau_n(\s, \ast) \supset \gamma^n_{(\s, \ast)}$$ with the following properties:
\begin{enumerate}[label=(\arabic*)] 
\item \label{item1} There exists an open interval of signed addresses $\Upsilon^n(\s,\ast) \ni (\s,\ast)$ such that $$\tau_n(\ul{\alpha}, \star) \subseteq \tau_n(\s, \ast) \quad \text{ for all } \quad (\ul{\alpha}, \star) \in \Upsilon^n(\s,\ast).$$
\item \label{item2} If $n\geq 1$, the restriction $f\vert_{\tau_n(\s, \ast)}$ is injective and maps to $\tau_{n-1}(\sigma(\s), \ast)$.	
\end{enumerate}
Hence, for all $n\geq 1$, we can define the inverse branch
\begin{equation}\label{eq_inversebranch}
f^{-1,[n]}_{(\s,\ast)}\defeq \left(f\vert_{ \tau_n(\s, \ast)}\right)^{-1} \colon f(\tau_n(\s, \ast)) \to \tau_n(\s, \ast).
\end{equation}
\end{thm} 

\begin{observation}[Chains of inverse branches {\cite[Observation 4.7]{mio_signed_addr}}]\label{obs_chain_inverse} Following Theorem~\ref{thm_inverse_hand}, for each $n\geq 0$ and $(\s, \ast)\in \Addr(f)_\pm$, we denote
$$f^{-n}_{(\s,\ast)}\defeq \left(f^n\vert_{ \tau_n(\s, \ast)}\right)^{-1} \colon f^n(\tau_n(\s, \ast)) \to\tau_n(\s, \ast).$$
Then, by Theorem~\ref{thm_inverse_hand}\ref{item2}, the following chain of embeddings holds:
$$ \tau_n(\s, \ast)\xhookrightarrow{ \; \;\;\; f \;\;\; \; } \tau_{n-1}(\sigma(\s), \ast)\xhookrightarrow{\; \;\; \; f\;\; \; \; } \tau_{n-2}(\sigma^2(\s), \ast) \xhookrightarrow{ \;\;\; \; f\;\;\; \; } \cdots \xhookrightarrow{\; \;\; \; f \; \; \;\; } \tau_{0}(\sigma^{n}(\s), \ast). $$
This means that we can express the action of $f^{-n}_{(\s, \ast)}$ in $f^n(\tau_n(\s, \ast))$ as a composition of functions defined in \eqref{eq_inversebranch}. That is, 
$$ \tau_n(\s, \ast)\xleftarrow{ f^{-1,[n]}_{(\s, \ast)} } f(\tau_n(\s, \ast))\xleftarrow{ f^{-1, [n-1]}_{(\sigma(\s), \ast)} } f^2(\tau_n(\s, \ast)) \xleftarrow{ f^{-1, [n-2]}_{(\sigma^2(\s), \ast)} } \cdots \xleftarrow{ f^{-1, [1]}_{(\sigma^{n-1}(\s), \ast)} } f^n(\tau_n(\s, \ast)). $$
More precisely,
$$f^{-n}_{(\s,\ast)} \equiv \left(f^{-1, [n]}_{(\s, \ast)}\circ f^{-1, [n-1]}_{(\sigma(\s), \ast)} \circ \cdots \circ f^{-1, [1]}_{(\sigma^{n-1}(\s), \ast)}\right)\!\Big\vert_{f^n(\tau_n(\s, \ast))}.$$
Moreover, combining this with Theorem~\ref{thm_inverse_hand}, for all $(\ul{\alpha}, \star) \in \Upsilon^n(\s,\ast)$,
$$f^n(\tau_n(\ul{\alpha}, \star))\subseteq f^n(\tau_n(\s, \ast)) \quad \text{ and } \quad f^{-n}_{(\s,\ast)}\big\vert_{f^n(\tau_n(\ul{\alpha}, \star))}\equiv f^{-n}_{(\ul{\alpha},\star)}\big\vert_{f^n(\tau_n(\ul{\alpha}, \star))}.$$ 
Finally, note that by Theorem~\ref{thm_signed}\ref{item:tails_bijection}, $f^{-n}_{(\s,\ast)}\colon \gamma^0_{(\sigma^n(\s), \ast)}\to \gamma^n_{(\s, \ast)}$ is a bijection.
\end{observation}
\section{Cantor bouquets and the class $\CB$}\label{sec_CB}

This section gathers the definitions and results from \cite{mio_newCB} that we need for our purposes. First, we adopt the definition of Cantor bouquet from \cite{AartsOversteegen,lasseBrushing}.
\begin{defn}[Straight brush, Cantor bouquet] \label{def_brush}
A subset $B$ of $[0,+\infty) \times (\R\setminus\Q)$ is a \emph{straight brush} if the following properties are satisfied:
\begin{itemize}
\item The set $B$ is a closed subset of $\R^2$.
\item For each $(x,y) \in B$, there is $t_y \geq 0$ so that $\{x: (x,y)\in B\} = [t_y, +\infty)$. 
The set $[t_y, +\infty) \times \{y\}$ is called the \textit{hair} attached at $y$, and the point $(t_y, y)$ is called its \textit{endpoint}.
\item The set $\{y: (x,y) \in B \text{ for some } x\}$ is dense in $\R\setminus\Q$. 
Moreover, for every $(x, y) \in B$, there exist two sequences of hairs attached respectively at $\beta_n, \gamma_n \in \R\setminus\Q$ such that $\beta_n < y < \gamma_n$, $\beta_n, \gamma_n \to y$ and $t_{\beta_n}, t_{\gamma_n} \to t_y$ as $n\to\infty$.
\end{itemize}
\noindent A \emph{Cantor bouquet} is any subset of the plane that is ambiently homeomorphic to a straight brush. A \textit{hair} (resp. \textit{endpoint}) of a Cantor bouquet is any preimage of a hair (resp. endpoint) of a straight brush under a corresponding ambient homeomorphism.
\end{defn}

\begin{observation}\label{obs_inverse_CB} Let $g\in \B$ be of disjoint type such that $J(g)$ is a Cantor bouquet. It follows from  \cite[Theorem 5.8]{lasse_arclike} that each hair $\eta$ of $J(g)$ is a dynamic ray together with its endpoint and that $g\vert_\eta$ is a bijection to another hair. Thus,  each hair of $J(f)$ but at most its endpoint belongs to $I(f)$; see \cite[Proposition~4]{mio_newCB} for details. 	
\end{observation}

In \cite{mio_newCB}, the following class of criniferous functions in $\B$ is introduced:
\begin{defn}[The class $\CB$] We say that $f\in \B$ belongs to the \textit{class $\CB$} if $J(\lambda f)$ is a Cantor bouquet for some\footnote{and thus all for which $\lambda f$ is of disjoint type, see \cite[Proposition 5]{mio_newCB}.} $\vert \lambda \vert $ sufficiently small.
\end{defn}

This natural class of criniferous functions includes all those that are a finite composition of functions of finite order in $\B$; see \cite[Proposition 6]{mio_newCB}. 

The following theorem gathers together all the results on function in $\CB$ we use in \S \ref{sec_semiconj}, and is a slightly modified version of \cite[Theorem~4.6]{mio_newCB}. In a broad sense, Theorem \ref{thm_CB} follows from the combination of two results: on one hand, \cite[Theorem~3.2]{lasseRigidity}  provides us with a conjugacy near infinity between any $f\in \B$ and a disjoint type $g(z)\defeq f(\lambda z)$. On the other hand, since $f\in \CB$, it is possible to project continuously the Cantor bouquet $J(g)$ to a forward invariant Cantor bouquet $X\subset J(g)$ where Rempe's conjugacy is defined, in such a way that the properties listed below hold. 
\begin{thm}\label{thm_CB} Let $f\in \CB$ and let $L>0$ so that $S(f)\subset \D_L$. Then, $f$ is criniferous, and there exists a disjoint type map $g(z)\defeq f(\lambda z)$ for some $\lambda\defeq \lambda(L) \in \C\setminus \{0\}$, with a Cantor bouquet $X\subset J(g)$, and a continuous map $\theta \colon J(g)\rightarrow J(f)$ with the following properties: 
\begin{enumerate}[label=(\alph*)]
\item \label{item:inclusions}$ \theta(J(g)) \cup J(g)\subset \C\setminus \D_L$;
\item \label{item:commute} $\theta \circ g =f\circ \theta $, except in a bounded set $B\subset \C\setminus X$;
\item \label{item:boundedC} there is a bounded set $C$ so that if $z\in B$, then the curve $[\theta(g(z)), f(\theta(z))]$ belongs to $C\cap f^{-1}(\C \setminus \D_L)$.
\item \label{item:homeomX} $\theta\vert_X$ is a homeomorphism onto its image;
\item \label{item:imageBouquet} $\theta(J(g))$ is a forward invariant Cantor bouquet so that $I(f)\subset \Orb^-(\theta(J(g)))$;
\item Each hair of $\theta(J(g))$ is either a ray tail or a dynamic ray with its endpoint;\label{item:rays}
\item  \label{item:annulus} there is 
$M\defeq M(L)>0$ such that for every $z \in J(g)$, $\theta(z) \in \overline{A(M^{-1}\vert z\vert, M\vert z\vert)}$; in particular, $\theta(I(g))\subset I(f)$;
\item \label{item:addr} the map $\theta$ establishes an order-preserving one-to-one correspondence between external addresses of $g$ and $f$;
\item \label{item:Ig} $I(g)\subset \Orb^-(X)$.
\end{enumerate}
\end{thm}
\begin{proof}
 That $f$ is criniferous is stated in \cite[Theorem~4.6 (e)]{mio_newCB}. In \cite[Definition~1.5]{mio_newCB}, a Cantor bouquet $X\subset J(g)$ is defined to be \textit{strongly absorbing} if, in particular, it is forward invariant and $I(g)$ is contained in its backwards orbit. The statement in \cite[Theorem~4.6]{mio_newCB} says that $X\subset J(g)$ is strongly absorbing, and so \ref{item:Ig} holds. Items \ref{item:inclusions}-\ref{item:addr} correspond to items \ref{item:inclusions}-\ref{item:addr} in \cite[Theorem~4.6]{mio_newCB}, noting that \ref{item:imageBouquet} is a consequence of \cite[Theorem~4.6(e)]{mio_newCB}, that  states that $\theta(J(g))$ is strongly absorbing.
\end{proof}

\begin{observation}\label{obs_CB_conf} By Theorem~\ref{thm_CB}\ref{item:imageBouquet},\ref{item:rays},\ref{item:addr}, the hairs of $\theta(J(g))$ are a valid initial configuration for $f$ in the sense of Definition \ref{def_initial_conf}.
\end{observation}

To conclude, we note that for any disjoint type map $g$ for which Theorem \ref{thm_CB}\ref{item:Ig} holds, it is possible to write the connected components of its escaping set as a nested union of preimages of those in the Cantor bouquet $X\subset J(g)$:

\begin{observation}\label{obs_betas}
Suppose that $\Addr(g)$ has been defined with respect to some alphabet of fundamental domains $\{F_i\}_{i\in I}.$ For each $\s=F_0 F_1 F_2\dots \in \Addr(g)$ and $n\in \N_{\geq 1}$, we denote 
\[ g_\s^n \defeq g\vert_{F_{n-1}}\circ g\vert_{F_{n-2}} \circ \cdots \circ g\vert_{F_{0}} \quad \text{ and } \quad  g^{-n}_\s\defeq \left(g^{n}_\s\right)^{-1}.\]
Let $X\subset J(g)$ be a Cantor bouquet such that $I(g)\subset \Orb^-(X)$, and for each $\s\in \Addr(g)$ and $n\geq 0$, let 
\begin{equation}\label{eq_Isbeta}
I_\s\defeq J_\s\cap I(g), \quad X_\s\defeq X\cap I_\s \quad \text{ and } \quad \beta^n_\s\defeq g_{\s}^{-n}(X_{\sigma^n(\s)}).
\end{equation}
Then, by Observation \ref{obs_addrz},
\[I_\s=\bigcup_{n\geq 0} \beta^n_\s \quad \text{ and } \quad g^n\colon \beta^n_\s \to X_{\sigma^n(\s)} \text{ is a bijection for all } n\geq 0.  \]
\end{observation}

\section{Strongly postcritically separated maps}\label{sec_postsep}

Recall that Theorem~\ref{thm_main_intro} holds for those maps in $\CB$ that are additionally strongly postcritically separated. The reason for it is that for the latter class of maps, postsingular points in their Julia set are ``sufficiently spread'' as to guarantee that the maps \textit{expand} an orbifold metric that sits on a neighbourhood of their Julia set, see Theorem~\ref{thm_orbifolds}. In this section, we include the main definitions and results from \cite{mio_orbifolds} that we require, and refer to that paper for more details.

Recall that for a holomorphic map $f:\widetilde{S}\rightarrow S$ 
between Riemann surfaces, the \emph{local degree} of $f$ 
at a point $z_0\in \widetilde{S}$, denoted by $\deg(f,z_0)$, is the unique integer $n\geq 1$ 
such that the local power series development of $f$ is of the form
\begin{equation*}
f(z)=f(z_0) + a_n (z-z_0)^n + \text{(higher terms)},
\end{equation*}
where $a_n\neq 0$. We say that $f$ has \textit{bounded criticality} on a set $A$ if $\AV(f) \cap A=\emptyset$ and there exists a constant $M<\infty$ such that $\operatorname{deg}(f,z)<M \text{ for all } z \in A.$

\begin{defn}[Strongly postcritically separated functions]\label{def_strongps}
A transcendental entire map $f$ is \textit{strongly postcritically separated} if there exist constants $c,\epsilon>0$ such that
\begin{enumerate}[label=(\alph*)]
\item \label{item_Fatou} $P(f)\cap F(f)$ is compact;
\item \label{itema_defsps} $f$ has bounded criticality on $J(f)$;
\item \label{itemb_defsps} for each $z\in J(f)$, $\#(\Orb^+(z)\cap \Crit(f)) \leq c$;
\item \label{itemd_defsps} for all distinct $z,w\in P_J\defeq P(f)\cap J(f)$, $\vert z-w\vert\geq \epsilon \max\{\vert z \vert, \vert w \vert\}$.
\end{enumerate}
\end{defn}

\begin{observation}[Dichotomy of points in $P_J$, {\cite[Observation 2.2]{mio_orbifolds}}] \label{rem_setting} If $f$ strongly postcritically separated, then any $p\in P_J$ is either (pre)periodic, or it escapes to infinity. If in addition $f\in \mathcal{B}$, then there can be at most finitely many points in $S(f)\cap I(f)$, and so $P(f)\setminus I(f)$ is bounded.
\end{observation}

Before stating our results on expansion, for the reader's convenience, we include some basic definitions on orbifold metrics. An \emph{orbifold} is a space that is locally represented as a quotient of an open subset $S$ of $\R^n$ by a linear action of a finite group (see \cite[Chapter 13]{thurstonGeom}). For the purposes of this paper, we are only interested in orbifolds modelled over Riemann surfaces. In this case, orbifolds are conveniently totally characterized by the surface $S$ together with a map that ``marks'' a discrete set of points of $S$; see \cite[Appendix A]{mcmullen1994complex} and \cite[Chapter 19 and Appendix E]{milnor_book} for a more detailed exposition. 
\begin{defn}[Riemann orbifold, covering orbifold maps]
A \emph{(Riemann) orbifold} is a pair $(S,\nu)$ consisting of a Riemann surface $S$, called the \textit{underlying surface}, and a \emph{ramification map} $\nu \colon S\rightarrow\N_{\geq 1}$ such that the set 
\begin{equation*}
\lbrace z\in S\; :\; \nu(z)>1\rbrace
\end{equation*}
is discrete. Let $\Ort=(\widetilde{S},\tilde{\nu})$ and $\Or=(S, \nu)$ be Riemann orbifolds. A \emph{holomorphic map} $f \colon\Ort\rightarrow\Or$ is a holomorphic map $f \colon\widetilde{S}\rightarrow S$ between the underlying Riemann surfaces such that
\begin{equation}\label{eq_holom_orb}
\nu(f(z)) \text{ divides } \deg(f,z)\cdot \tilde{\nu}(z)\quad \text{ for all } \quad z\in \widetilde{S}.
\end{equation}
If in addition $f \colon\widetilde{S}\rightarrow S$ is a branched covering map such that 
\begin{equation}\label{eq_covering_orb}
\nu(f(z)) = \deg(f,z)\cdot\tilde{\nu}(z)\quad \text{ for all } \quad z\in \widetilde{S},
\end{equation}
then $f \colon\Ort\rightarrow\Or$ is an \emph{orbifold covering map}. If there exists an orbifold covering map between $\Ort$ and $\Or$, $\widetilde{S}$ is simply-connected and $\tilde{\nu}\equiv 1$, then $\Ort$ a \emph{universal covering orbifold} of $\Or$ and $f$ is a \textit{universal covering map}.
\end{defn} 

\begin{remark}With a slight abuse of notation, for points or sets $z, A\subset \C$ we will sometimes write $z, A \in \Or $ to indicate that $z$ or $A$ belong to the underlying surface of $\Or$.
\end{remark}

As a generalization of the Uniformization theorem for Riemann surfaces, with only two exceptions, every Riemann orbifold has a universal covering orbifold: 
\begin{thm}[Uniformization of Riemann orbifolds]\label{thm_uniform}
Let $\Or=(S,\nu)$ be a Riemann orbifold for which $S$ is connected. Then $\Or$ has no universal 
covering orbifold if and only if $\Or$ is isomorphic to $\widehat{\C}$ 
with signature $(l)$ or $(l,k)$, where $l\neq k$. In all other cases the 
universal cover is unique up to a conformal isomorphism over the surface $S$, and given by either $\widehat{\C}$, $\C$ or $\D$. In particular, if $S\subsetneq \C$ and $\#(\widehat{\C} \setminus S)>2$, then $\Or$ is covered by $\D$.
\end{thm}

In analogy to Riemann surfaces, we call an orbifold $\Or$ \emph{elliptic, parabolic} or \emph{hyperbolic} if all of its connected components are covered by $\widehat{\C}, \C$ or $\D$ respectively. In this paper we are only interested in hyperbolic orbifolds.

\begin{discussion}[Orbifold metric and distance] \label{orb_metric}
Theorem~\ref{thm_uniform} allows us to induce a metric on those orbifolds that have a universal cover as the pushforward of the spherical, Euclidean or hyperbolic metric of their universal cover. For our purposes, let $\Or=(S, \nu)$ be an orbifold that has universal covering surface $\D$, and let $\rho_\D(z)\vert dz\vert$ be the hyperbolic metric in $\D$. By pushing forward this metric by an orbifold covering map, we obtain a Riemannian metric on $\Or$, that we denote by $\rho_{\Or}(w)\vert dw\vert$ and call the \emph{orbifold metric} of $\Or$. The orbifold metric on $\Or$ determines a metric in the surface $S$ with singularities at the ramified points of $\Or$. More precisely, if $\nu(w_0)=m>1$ for some $w_0 \in S$, then $\rho_{\Or}(w)\vert dw\vert$ has a singularity of the type $\vert w-w_0\vert^{(1-m)/m}$ near $w_0$ in $S$.

This metric induces an \textit{$\Or$-distance} between points $x,y \in S$ in the following way. We join $x$ to $y$ by a rectifiable curve $\gamma$ in $S$, and define
the $\Or$-length $\ell_\Or(\gamma)$ of $\gamma$ by
$$\ell_\Or(\gamma) \defeq \int_\gamma \rho_{\Or}(w)\vert dw\vert.$$
Note that the integral is well-defined, since the set of ramified points in $\gamma$, and thus singularities of $\rho_{\Or}$, is finite. Finally, we set
$$d_\Or(x, y) \defeq \inf\{\ell_\Or(\gamma): \gamma \text{ is a rectifiable curve in } S \text{ joining } x \text{ and } y \}.$$
\end{discussion}

The following theorem gathers the main results on strongly postcritically separated maps that we will use, and is a compendium of \cite[Definition and Proposition 5.1, Theorem~1.1 and Lemma~5.3]{mio_orbifolds}.
\begin{thm}[Orbifold expansion for strongly postcritically separated maps] \label{thm_orbifolds} Let $f\in \B$ be a strongly postcritically separated map. Then, there exist hyperbolic orbifolds $\Ort=(\widetilde{S},\tilde{\nu})$ and $\Or=(S,\nu)$ with the following properties:
\begin{enumerate}[label=(\alph*)]
\item \label{item_a_assocorb} Either $S=\C =\tilde{S}$, or $\text{cl}(\tilde{S})\subset S=\C\setminus \overline{U}$, where $U$ is a finite union of bounded Jordan domains. Moreover, $J(f) \subset S$.
\item  \label{item_e_assocorb} $f \colon\Ort\to\Or$ is an orbifold covering map.
\item There exists a constant $\Lambda > 1$ such that
\begin{equation}\label{eq_exp_lambda}
\Vert \Deriv f(z)\Vert_{\Or}\defeq\frac{\vert f'(z)\vert\rho_{\Or}(f(z))}{\rho_{\Or}(z)} \geq\Lambda,
\end{equation}
whenever the quotient is defined.
\item \label{lem_annulus} For any fixed $K>1$, there is $R>0$ so that if $p,q\in \overline{A(t, Kt)}\subset A(t/K,t K^2) \subset S$ for some $t>0$, then $d_\Or(p,q) \leq R$. 
\end{enumerate}		
\end{thm}
\begin{cor}[Shrinking of preimages of bounded curves {\cite[Corollary 5.4]{mio_orbifolds}}] \label{cor_uniform}  Under the assumptions of Theorem~\ref{thm_orbifolds}, for any curve $\gamma_0 \subset \Or$, for all $k\geq 1$, and each curve $\gamma_k \subset f^{-k}(\gamma_0)$ such that $f^k\vert_{\gamma_k}$ is injective, $\ell_{\Or}(\gamma_k)\leq \frac{\ell_{\Or}(\gamma_0)}{\Lambda^{k}}$ for some constant $\Lambda>1$.	
\end{cor}
Note that the preceding corollary is meaningful only when applied to some $\gamma_0$ of bounded orbifold length. We might not be able to guarantee so in all cases we wished, so, instead, we will consider rectifiable curves in the same ``sort of homotopy class'' in the following sense. 
\begin{discussion}[Definition of the sets $\mathcal{H}^{q}_{p} (W(k))$] Let us fix an entire function $f$ and let $k\in \N$. We suggest the reader keeps in mind the case when $k=0$, since it will be the one of greatest interest for us. Let $W(k)$ be a finite set of (distinct) points in $f^{-k}(P(f))$, totally ordered with respect to some relation ``$\prec$''. That is, $W(k)\defeq(W(k), \prec)=\{w_1, \ldots, w_N \}\subset f^{-k}(P(f))$ such that $w_{j-1}\prec w_j\prec w_{j+1}$ for all $ 1< j< N$. We note that $W(k)$ can be the empty set. Then, for every pair of points\footnote{In particular, $p$ and $q$ might belong to $f^{-k}(P(f))$.} $p,q \in \C\setminus W(k)$, we denote by $\mathcal{H}^{q}_{p} (W(k))$ the collection of all curves in $\C$ with endpoints $p$ and $q$ that join the points in $W(k)$ \textit{in the order} ``$\prec $'', starting from $p$. More formally, $\gamma\in \mathcal{H}^{q}_{p} (W(k))$ if $\text{int}(\gamma)\cap f^{-k}(P(f))=W(k)$ and $\gamma$ can be parametrized so that $\gamma(0)=p$, $\gamma(1)=q$ and $\gamma(\frac{j}{N+1})=w_j$ for all $1\leq j \leq N$. In particular, $\gamma$ can be expressed as a concatenation of $N+1$ curves 
\begin{equation}\label{def_concat}
\gamma=\gamma^{w_{1}}_{p} \bm{\cdot}\gamma^{w_{2}}_{w_1} \bm{\cdot} \cdots \bm{\cdot}\gamma^{q}_{w_N},
\end{equation}
each of them with endpoints in $W(k) \cup\{p,q\}$ and such that $$\text{int}(\gamma^{w_1}_{p}), \text{int}(\gamma^{w_{i+1}}_{w_i}), \text{int}(\gamma^{q}_{w_N}) \subset \C \setminus f^{-k}(P(f))$$
for each $1\leq i \leq N$.
\end{discussion}

\noindent We use the following notion of homotopy for the sets of curves described:
\begin{defn}[Post-$k$-homotopic curves] \label{def_post0}
Consider $W(k)=\{w_1, \ldots, w_N \}\subset f^{-k}(P(f))$ and two curves $\gamma,\beta\in \mathcal{H}^{w_{N+1}}_{w_0}(W(k))$, for some $\{w_0, w_{N+1}\}\subset \C\setminus W(k)$. We say that $\gamma$ is \emph{post-$k$-homotopic to $\beta$} if for all $0 \leq i\leq N$, $\gamma^{w_{i+1}}_{w_i}$ is homotopic to $\beta^{w_{i+1}}_{w_i} \text{ in } (\C \setminus f^{-k}(P(f))) \cup \{w_i, w_{i+1} \}$.
\end{defn}

In other words, for each $1\leq i\leq N$, the restrictions of $\gamma$ and $\beta$ between $w_i$ and $w_{i+1}$ are homotopic in the space $(\C \setminus f^{-k}(P(f))) \cup \{w_i, w_{i+1}\}$. It is easy to see that this defines an equivalence relation in $\mathcal{H}^q_p (W(k))$, with $p=w_0$ and $q=w_{N+1}$. For each $\gamma\in \mathcal{H}^q_p (W(k))$, we denote by $[\gamma]_{_k}$ its equivalence class. Note that if $W(k)=\emptyset$ and $p,q \in \C \setminus f^{-k}(P(f))$, then for any curve $\gamma\in\mathcal{H}^q_p (W(k))$, $[\gamma]_{_k}$ equals the equivalence class of $\gamma$ in $\C \setminus f^{-k}(P(f))$ in the usual sense. Moreover, if $\gamma$ is any curve that meets only finitely many elements of $f^{-k}(P(f))$, then it belongs to a unique set of the form ``$\mathcal{H}^q_p(W(k))$'' up to reparametrization of $\gamma$, and so its equivalence class $[\gamma]_{_k}$ is defined in an obvious sense. Hence, the notion of post-$k$-homotopy is well-defined for all such curves.

Recall that if $f$ is an entire function, for any two curves $\gamma, \beta \subset f(\C)\setminus P(f)$, homotopic relative to their endpoints, by the \textit{homotopy lifting property}, see \cite{hatcher2002algebraic}, for each curve in $f^{-1}(\gamma)$, there exists a curve in $f^{-1}(\beta)$ homotopic to it relative to their endpoints, since $f$ acts as a covering map. The following is an analogue of the Homotopy Lifting Property for post-$k$-homotopic curves.
\begin{prop}[Post-homotopy lifting property {\cite[Proposition 7.3]{mio_orbifolds}}]\label{cor_homot}
Let $f$ be an entire map and let $C\subset \C$ be a domain so that $f^{-1}(C) \subset C$ and $\AV(f) \cap C=\emptyset.$ Let $\gamma \subset C$ be a bounded curve such that $\#(\gamma \cap P(f))< \infty$. Fix any $k\geq 0$ and any curve $\gamma_k\subset f^{-k}(\gamma)$ for which the restriction $f^k\vert_{ \gamma_k}$ is injective. Then, for each $\beta \in [\gamma]_{_0}$, there exists a unique curve $\beta_k\subset f^{-k}(\beta)$ such that $\beta_k \in [\gamma_k]_{_k}$. In particular, $\beta_k$ and $\gamma_k$ share their endpoints.
\end{prop}

Finally, the next result tells us that if $f$ is an entire function that has dynamic rays, and $U$ is a certain subdomain of any hyperbolic orbifold, then there exists a constant $\mu$ such that for every piece of dynamic ray contained in $U$, we can find a curve in its post-$0$-homotopy class with orbifold length at most $\mu$. This result will be of great use to us for the following reason: in Lemma~\ref{lem_UHSC}, rather than pulling back pieces of dynamic rays that might not be rectifiable, we will instead pull-back curves in their same post-$0$-homotopy class, for which, by Theorem~\ref{cor_homot2}, we have a uniform bound on their length.

\begin{thm}[Pieces of rays with uniformly bounded length {\cite[Corollary 7.8]{mio_orbifolds}}]\label{cor_homot2}
Let $f\in \B$, let $\Or=(S, \nu)$ be a hyperbolic orbifold with $S\subset \C$, and let $U\Subset S$ be a simply connected domain with locally connected boundary. Assume that $P(f)\cap \overline{U} \subset J(f)$, $\#( P(f)\cap \overline{U})$ is finite and there exists a dynamic ray or ray tail landing at each point in $P(f)\cap U$. Then, there exists a constant $L_U \geq 0$, depending only on $U$, such that, for any (connected) piece of ray tail $\xi \subset U$, there exists $\delta\in [\xi]_{_0}$ with $\ell_{\Or}(\delta)\leq L_U.$
\end{thm}
\section{The model space}\label{sec_model}

Recall that our goal is to define a model for the action of $f$ on $J(f)$, Theorem~\ref{thm_main_intro}. The potential presence of critical values, and hence critical points, in $I(f)$, suggests the use of two copies of $J(g)$, say $J(g)\times \{-,+\}$, as a candidate model space: then, in a very rough sense, a function $\phi$ could map $J(g)\times \{-,+\}$ to $J(f)$ by mapping $J(g)\times \{+\}$ to the closure of those canonical rays of the form $\Gamma(\cdot, +)$, and $J(g)\times \{ -\}$ to the closure  of those canonical rays of the form $\Gamma(\cdot, -)$. We will proceed this way in \S\ref{sec_semiconj}.

Note that for the function $\phi$ to be continuous, we need to provide the set $J(g)\times \{-,+ \}$ with a topology that is compatible with that of $\Addr(f)_\pm$ defined in \ref{discussion_signed_addr}. This is our main task in this section. Even if a topology could be defined directly over $J(g)\times \{-,+ \}$, for convenience and simplification of arguments, we instead define it in two copies of any straight brush $B$, and using the corresponding ambient homeomorphism $\psi\colon J(g) \rightarrow B$, we induce a topology in our model set in \ref{dis_topology_model}.

For the rest of the section, let us fix any $f\in \CB$ and a disjoint type function $g$ from its parameter space. Let us moreover assume that the topological space $\Addr(g)_\pm$ has been defined following Definition \ref{defn_signedaddr}. Recall from Definition \ref{def_brush} that a straight brush is defined as a subset of $[0,\infty)\times \R \setminus \Q $. Hence, we consider the set
\begin{equation}\label{eq_setM}
\M\defeq [0,\infty)\times \R \setminus \Q \times \lbrace -,+\rbrace,
\end{equation}
that we aim to endow with a topology. We will use the symbols ``$\ast, \star, \circledast$'' to refer to generic elements of $\lbrace -,+\rbrace$. 

\begin{discussion}[Topology in $(\R \setminus \Q) \times \lbrace -,+\rbrace$]\label{dis_modelM_topIrr}
We start by providing $(\R \setminus \Q) \times \lbrace -,+\rbrace$ with a topology compatible with that of $\Addr(g)_\pm$. Let $<_{_i}$ be the usual linear order on irrationals, and let us give the set $\lbrace -,+\rbrace$ the order $\lbrace -\rbrace \prec \lbrace +\rbrace$. Then, for elements in the set $(\R \setminus \Q) \times \lbrace -,+\rbrace$, we define the order relation 
\begin{equation}\label{eq_orderB}
(r, \ast )< (s, \star) \qquad \text{ if and only if } \qquad r <_{_i} s \quad \text{ or } \quad r = s \: \text{ and } \: \ast \prec \star.
\end{equation} 
This gives a total order to $(\R \setminus \Q) \times \lbrace -,+\rbrace$. Thus, we can define a cyclic order induced by ``$<$'' in the usual way: for $a,x,b \in (\R \setminus \Q) \times \lbrace -,+\rbrace$,
$$[a,x,b]_{_I} \quad \text{if and only if} \quad a< x< b \quad \text{ or } \quad x < b < a \quad \text{ or }\quad b < a < x.$$
Moreover, given two different elements $a,b \in (\R \setminus \Q) \times \lbrace -,+\rbrace$, we define the \textit{open interval} from $a$ to $b$, denoted by $(a,b)$, as the set of all points $x\in (\R \setminus \Q)\times \lbrace -,+\rbrace$ such that $[a,x,b]$. The collection of all such open intervals forms a base for the \textit{cyclic order topology}, that we denote by $\tau_I$.
\end{discussion}

Before we proceed to define a topology in $\mathcal{M}$, let us check that the topological spaces $(\Addr(g)_\pm,\tau_A)$ and $(\R \setminus \Q), \tau_I)$ are indeed closely related.
\begin{prop}[Correspondence between spaces]\label{prop_mapC} Let $\psi\colon J(g)\rightarrow B$ be an ambient homeomorphism and for each $\s \in \Addr(g)$, let $\Irr(\s)\defeq y,$ where $y$ is the irrational so that $\psi(J_\s)=[t_y, \infty) \times \{y\}\subset B$. Let $\mathcal{C}\colon\Addr(g)_\pm \rightarrow (\R \setminus \Q)\times \{-,+ \}$ given by $\mathcal{C}(\ul{s}, \ast)=(\Irr(\s),\ast)$. Then $\mathcal{C}$ is an open map.
\end{prop}
\begin{proof} Let $\ul{s}, \ul{\tau}, \ul{\alpha} \in \Addr(g)$. Let $[\cdot]_\ell$ denote the lexicographic cyclic order on $\Addr(g)$ as described in \ref{discussion_signed_addr}. Then,
$$[\ul{s}, \ul{\tau}, \ul{\alpha}]_{_\ell} \xLeftrightarrow{ \; \; _{(1)} \; \;} [ J_\s, J_{\ul{\tau}}, J_{\ul{\alpha}}]_\infty \xLeftrightarrow{ \; \; _{(2)} \; \;} [
\psi (J_\s),\psi (J_{\ul{\tau}}), \psi (J_{\ul{\alpha}})]_\infty \xLeftrightarrow{ \; \; _{(3)} \; \;}[\Irr(\ul{s}),\Irr( \ul{\tau}),\Irr( \ul{\alpha})]_{_i},$$ 
where (1) is \cite[Observation 2.14]{mio_signed_addr}, ${(2)}$ is by $\psi$ being a homeomorphism and hence preserving the cyclic order at infinity, and ${(3)}$ is by defining a cyclic order in the irrationals from the usual linear order. Then, if we respectively cut the cyclic orders $[\cdot]_\ell$ and $[\cdot]_i$ in some external address $\s$ and $\Irr(\s)$, since the linear orders in $\Addr(g)_\pm$ and $(\R \setminus \Q)\times \{-,+ \}$ are respectively defined in \eqref{eq_linear_addr} and \eqref{eq_orderB} the same way, it follows that
\begin{equation}\label{eq_mapC}
[(\s, \ast),(\underline{\alpha}, \star),(\underline{\tau},\circledast)]_{_A} \quad \text{if and only if} \quad [\mathcal{C}((\s, \ast)),\mathcal{C}((\underline{\alpha}, \star)),\mathcal{C}((\underline{\tau},\circledast))]_{_I}.
\end{equation}
Then, since we have used these orders to define the respective cyclic order topologies $\tau_A$ and $\tau_I$ in the respective domain and codomain of $\mathcal{C}$, we have that $\mathcal{C}$ is an open map.
\end{proof}

We observe some properties of the topological space defined in \ref{dis_modelM_topIrr}.
\begin{observation}[Open and closed sets in $(\R \setminus \Q \times \{-,+\}, \tau_I)$]\label{rem_cyclicIrr}
Let $A$ be an open set of $(\R \setminus \Q \times \{-,+\}, \tau_I)$ and suppose that $(s,-), (s,+)\in A$. Then, since $ \tau_I$ is generated by open intervals, there exist irrationals $r<_{_i}s<_{_i}t$ such that $((r,\ast), (s,+))\ni (s,-)$ and $((s,-),(t,\ast)) \ni (s,+)$. Hence, both $(s,-), (s,+) \in((r,\ast), (t,\ast)) \subset A$. Moreover, for any pair $r<_{_i} s$, the sets $U\defeq ((r, +), (s,-))$, $U \cup \{(r,+)\}$, $U \cup \{(s,-)\}$ and $U \cup \{(r,+),(s,-)\}=((r, -), (s, +)) \eqdef V$ are open intervals. In addition, $V$ is also closed, since it contains its boundary points.
\end{observation}

\begin{discussion}[Definition of topologies]\label{dis_topology_model}
Let $\mathcal{M}$ be the set from \eqref{eq_setM}. We define the topological space $(\M, \tau_\M)$, with $\tau_\M$ being the product topology of $[0,\infty)$ with the usual topology, and $(\R \setminus \Q)\times \lbrace -,+\rbrace$ with the topology $\tau_I$. Let $B$ and $\psi$ be a straight brush and usual ambient homeomorphism for which $\psi(J(g))=B$. Let $B_{\pm}\defeq B \times \lbrace -,+\rbrace$ be the subspace of $\M$ with the induced topology $\tau_{B_{\pm}}$ from $\tau_\M$. Consider the set $J(g)_{\pm}\defeq J(g)\times \lbrace -,+\rbrace$ and the bijection $\tilde{\psi}\colon J(g)_{\pm} \to B_\pm$ defined as $\tilde{\psi}((z,\ast))\defeq (\psi(z),\ast)$. We can then induce a topology in $J(g)_{\pm}$ from the space $(B_\pm,\tau_{B_{\pm}})$, namely 
\begin{equation}\label{eq_topologyJ}
\tau_J\defeq \lbrace \tilde{\psi}^{-1}(U) \colon U\in \tau_{B_\pm} \rbrace.
\end{equation}
Note that in particular, $\tilde{\psi}\colon (J(g)_{\pm}, \tau_J) \to (B_\pm,\tau_{B_{\pm}})$ is a homeomorphism. We moreover define $I(g)_{\pm}\defeq I(g)\times \lbrace -,+\rbrace \subset J(g)_{\pm}$ as a subspace equipped with the induced topology. 
\end{discussion}

\begin{defn}[Model for functions in $\CB$]\label{defn_modelspace}
Let $f\in \CB$ and let $g$ be any disjoint type function on its parameter space. Then, the space $(J(g)_\pm, \tau_J)$, with $\tau_J$ defined following \ref{dis_topology_model}, is a \textit{model space} for $f$. Moreover, we define its \textit{associated model function} $\tilde{g} \colon J(g)_\pm\rightarrow J(g)_\pm$ as $\tilde{g}(z, \ast)\defeq(g(z),\ast)$. 
\end{defn}	
	
\begin{observation}[All models for a fixed function are conjugate]\label{obs_unique_model} Let $f\in \CB$ and let $g_1$ and $g_2$ be two disjoint type functions on its parameter space. Let $J(g_1)_\pm$, $J(g_2)_\pm$ and $\tilde{g}_1$, $\tilde{g}_2$ be the corresponding models and respective associated model functions. Then, there exists a homeomorphism $\Phi\colon J(g_1)_\pm \rightarrow J(g_2)_\pm$ such that $\Phi\circ \tilde{g}_1=\tilde{g}_2 \circ \Phi$. To see this, note that since any two straight brushes are ambiently homeomorphic, we may assume without loss of generality that the topologies in $J(g_1)_\pm$ and $J(g_2)_\pm$ have been induced from the same space $B_\pm$ following \ref{dis_topology_model}. It follows from \cite[Proof of Theorem 3.1]{lasseRigidity}, see \cite[Corollary~2]{mio_newCB}, that there exists a homeomorphism $\phi\colon J(g_1)\rightarrow J(g_2)$ such that $\phi\circ g_1=g_2 \circ \phi$. Defining $\Phi(z, \ast)\defeq(\phi(z),\ast)$, our claim follows.
\end{observation}	
		
\begin{remark}
The space $(\M,\tau_\M)$ is not second countable, see \cite[remark on p.124]{mio_thesis}, and so it cannot be (topologically) embedded on the plane. Similarly, nor can $(J(g)_{\pm}, \tau_J)$. Nonetheless, consider any open set $U$ of $\M$ of the form $U \defeq (t_1, t_2) \times (x,y)$, with $t_1<t_2$ and $x=(r,\ast) ,\: y=(s,\star) \in (\R \setminus \Q) \times \lbrace -,+\rbrace$ for some $r<_{_i}s$. Then, the interval $(x,y)$ comprises all elements $(\alpha, \ast)$ with $r<_{_i}\alpha<_{_i}s$, and so, we can think of $U$ as being a sort of ``box''. This intuition will become clearer in the proof of the next proposition.
\end{remark}

\begin{prop}[Continuity of functions from the model space]\label{prop_contmodel} Let $f\in \CB$ and let $J(g)_\pm$ be a model space for $f$. Then, both its associated model function $\tilde{g}$ and the function $\pi\colon J(g)_{\pm} \rightarrow J(g)$ given by $\pi(z, \ast)\defeq z$ are continuous.
\end{prop}
\begin{proof}
Let $\tilde{\psi}\colon J(g)_{\pm} \to B_\pm$ and $\psi \colon J(g)\rightarrow B$ be the homeomorphisms from \ref{dis_topology_model}. Since proving continuity of $\pi$ is equivalent to proving continuity of $\mathcal{P}\defeq (\psi\circ \pi \circ \tilde{\psi}^{-1})\colon B_\pm \rightarrow B$, we do the latter. For any $(t,r,\ast)\in B_\pm$,
\begin{equation} \label{eq_P}
\mathcal{P}(t,r,\ast)=(\psi\circ\pi\circ\tilde{\psi}^{-1})(x)= (\psi\circ\pi)(\psi^{-1}(x), \ast)=(\psi\circ\psi^{-1})(x)=(t,r).
\end{equation}
Fix $x=(t,r,\ast)\in B_\pm$, any $\epsilon>0$ and let $\D_{\epsilon}(t,r)$ be the (Euclidean) ball of radius $\epsilon$ centred at $\mathcal{P}(x)$. We can find a pair of irrational numbers $r_1<r<r_2$ such that the rectangle $(t-\epsilon/2,t+\epsilon/2 )\times (r_1, r_2) \subset \D_{\epsilon}(t,r)$. Then, $R\defeq ((t-\epsilon/2,t+\epsilon/2) \times ((r_1,+), (r_2,-))\cap B_\pm)$ is an open subset of $B_\pm$ containing $x$ and such that 
$$\mathcal{P}\left(R\right)=(t-\epsilon/2,t+\epsilon/2 )\times (r_1, r_2)\subset \D_{\epsilon}(t,r),$$
and so $\mathcal{P}$ is continuous. Similarly, proving that $\tilde{g}\colon J(g)_\pm \rightarrow J(g)_\pm$ is continuous is equivalent to proving that $\tilde{h}\defeq \tilde{\psi}\circ \tilde{g}\circ \tilde{\psi}^{-1}\colon B_\pm \rightarrow B_\pm$ is continuous. For any $x=(t,r,\ast)\in B_\pm$,
$$\tilde{h}(x)=(\tilde{\psi}\circ\tilde{g}\circ\tilde{\psi}^{-1})(x)=(\tilde{\psi}\circ\tilde{g})(\psi^{-1}(t,r), \ast)=((\psi\circ g \circ\psi^{-1})(t,r), \ast).$$
That is, $\tilde{h}(t,r,\ast)= (h(t,r), \ast)$, where $h\defeq \psi\circ g\circ\psi^{-1}\colon B \rightarrow B$ is a continuous function.

Fix $x\in B_\pm$ and let $V_x$ be an open neighbourhood of $\tilde{h}(x)\eqdef(t,\alpha,\ast)$. We may assume without loss of generality that $V_x$ is of the form $V_x\defeq (t_1,t_2) \times (w,y)$, with $t_1< t <t_2\in \R$ and $w=(r, \circledast), y=(s,\star) \in B_\pm$ such that $r\leq_{_i} \alpha \leq_{_i} s$. Let $\mathcal{P}\colon B_\pm \rightarrow B$ be the function specified in \eqref{eq_P}. If, $r \lneq_{_i}\alpha \lneq_{_i} s$, then $(t,\alpha, -), (t,\alpha, +) \in V_x$, and by Observation~\ref{rem_cyclicIrr}, there exists a pair of irrationals $\alpha^-, \alpha^+$ so that $r\leq \alpha^-<\alpha<\alpha^+\leq s$ and $H\defeq (t_1,t_2) \times ((\alpha^-,+), (\alpha^+,-)) \subset V_x$. In particular, $\mathcal{P}(H)$ is open and $(\mathcal{P}^{-1}\circ\mathcal{P})(H)=H$. Since both $h$ and $\mathcal{P}$ are continuous functions, $X\defeq (\mathcal{P}^{-1}\circ h^{-1}\circ \mathcal{P})(H)$ is an open set in $B_\pm$, and by construction, $X$ is a neighbourhood of $x$ such that $\tilde{h}(X)\subset V_x$. Otherwise, either $r=\alpha$, which implies that for $V_x$ being an open neighbourhood of $\tilde{h}(x)$, $w$ must be of the form $w=(r, -)$ and $\tilde{h}(x)=(t, r, +)$, or by the same reasoning, $y=(s,+)$ and $\tilde{h}(x)=(t,s,-)$. We only argue continuity in the first case and remark that the second case can be dealt with analogously. Define $R\defeq (t_1,t_2) \times (r, -)$ and $H\defeq (t_1,t_2) \times ((r,+), (s,\star)) \subset V_x$. Note that $\mathcal{P}(H)$ is an open set and $\mathcal{P}(R)\subset \partial \mathcal{P}(H)$. Since $g$ is of disjoint type, $J(g) \cap \Crit(g)=\emptyset$, and so $g$ is locally injective in $J(g)$. Therefore, so is $h$, which implies that $h$ preserves locally the order of the hairs of the straight brush $B$. Consequently, $(h^{-1}\circ\mathcal{P})(R) \subset \partial (h^{-1} \circ\mathcal{P})(H).$ By construction and a similar argument to that in the previous case, the set $(h^{-1}\circ\mathcal{P})(R)\times \{-\}\cup (\mathcal{P}^{-1}\circ h^{-1} \circ \mathcal{P})(H)$ is an open neighbourhood of $x$ whose image under $\tilde{h}$ lies in $V_x$, and continuity follows.
\end{proof}

We conclude this section  with some topological properties of model spaces that will be of use to us in section \ref{sec_semiconj} when proving surjectivity of the function $\phi$ from Theorem~\ref{thm_main_intro}.
\begin{lemma}[Compactification of the model space]\label{compactification} Let $f\in \CB$ and let $J(g)_\pm$ be a model space for $f$. Then, $J(g)_\pm$ admits the one point (or Alexandroff)-compactification $\tau_\infty$. The new compact space $(J(g)_{\pm}\cup\lbrace\tilde{\infty}\rbrace,\tau_\infty)$ is a sequential space. Moreover, given a sequence $\lbrace x_k \rbrace_{k\in \N} \subset J(g)_{\pm}\cup\lbrace\tilde{\infty}\rbrace$, 
\begin{equation}\label{limits}
 \lim_{k \rightarrow \infty} x_k=\tilde{\infty} \quad \Longleftrightarrow \quad \lim_{k \rightarrow \infty} \pi(x_k)=\infty. 
\end{equation} 
\end{lemma}
\begin{proof}
We show that $J(g)_{\pm}$ admits a one-point compactification by proving that $J(g)_{\pm}$ is a locally compact, Hausdorff space. Equivalently, since these are topological properties (preserved under homeomorphisms), we instead show that a corresponding double brush $(B_\pm,\tau_{B_{\pm}})$, see \ref{dis_topology_model}, is locally compact and Hausdorff. Note that the space $(\M,\tau_{\M})$ defined in \ref{dis_topology_model} is Hausdorff: for any $(t,s)\in \R^2$, 
\begin{equation*}
\begin{split}
(t,s,-) \in V_{-}\defeq (t-t/2, t+1)\times((s-1,-),(s,+)), \\ (t,s,+) \in V_{+}\defeq (t-t/2, t+1)\times ((s,-),(s+1,+)),
\end{split}
\end{equation*}
and $V_{-} \cap V_+=\emptyset$. Disjoint neighbourhoods of any pair of points in $\M$ can be constructed similarly. Since being Hausdorff is a hereditary property, $(B_\pm,\tau_{B_{\pm}})$ is Hausdorff.

We prove local compactness of $(B_\pm,\tau_{B_{\pm}})$ by showing that for each $x \in B_\pm$ and each open bounded neighbourhood $U_x \ni x$, the closure of $U_x$ in $B_\pm$ , that we denote by $\overline{U}_x$, is compact. With that purpose, let $\mathcal{U}=\lbrace{U_i}\rbrace_{i \in I}$ be an open cover of $\overline{U}_x$. By definition, $B_\pm\setminus \overline{U}_x$ is an open set, and so 
$$ \mathcal{U}'\defeq \lbrace{U_i}\rbrace_{i\in I} \cup \lbrace B_\pm\setminus \overline{U}_x \rbrace$$
is an open cover of $B_\pm$. Hence, for each $(t,s, \ast) \in B_\pm$, there exists $U_{(t,s, \ast)} \in \mathcal{U}'$ such that $(t,s,\ast)\in U_{(t,s, \ast)}.$ For each $(t,s) \in B$, denote 
$$V_{(t,s)}\defeq U_{(t,s, -)} \cup U_{(t,s, +)} \quad \text{ and }\quad \mathcal{V}\defeq \{V_{(t,s)}\}_{(t,s) \in B}.$$
Let $\mathcal{P}\colon B_\pm \rightarrow B$ be the projection function specified in \eqref{eq_P}, and observe that $\mathcal{P}(V_{(t,s)})$ might not be open, but since both $\{(t,s, -), (t,s, +) \}\subset V_{(t,s)}$, by Observation \ref{rem_cyclicIrr}, $\mathcal{P}(V_{(t,s)})$ always contains an open neighbourhood $W_{(t,s)} \ni (t,s),$ that we take to be $\mathcal{P}(V_{(t,s)})$ when this set is open. Then, 
$$\mathcal{W}\defeq \lbrace W_{(t,s)} \rbrace_{(t,s)\in B}$$
forms an open cover of $B\subset \R^2$, and in particular of the closure $\overline{\mathcal{P}(\overline{U}_x)}$. Note that $\mathcal{P}(\overline{U}_x)$ is a bounded set, since $\overline{U}_x$ is bounded and we showed in the proof of Proposition~\ref{prop_contmodel} that $\mathcal{P}$ is continuous. Since the straight brush $B\subset \R^{2}$ satisfies the Heine-Borel property, there exists a finite subcover $\tilde{\mathcal{W}}=\lbrace{W_k}\rbrace_{k \in K} \subset \mathcal{W}$ of $\overline{\mathcal{P}(\overline{U}_x)}$. For each $k\in K$, choose $V_k \in \V$ such that $W_k \subseteq \mathcal{P}(V_k)$ and denote $$\tilde{\mathcal{V}}\defeq \{V_k\}_{k \in K} \quad \text{ and } \quad \tilde{ \mathcal{U}} \defeq \lbrace U_{(t,s, \ast)} \in \mathcal{U} \colon U_{(t,s, \ast)} \subset V_{k} \in \tilde{ \mathcal{V}}\rbrace.$$
By definition, the set $\tilde{ \mathcal{V}}$ has the same number of elements as ${\tilde{\mathcal{W}}}$ does, and $\tilde{ \mathcal{U}}$ has at most double, and so a finite number. Thus, if we show that $\tilde{\mathcal{U}}$ is an open subcover of $\overline{U}_x$, then we will have shown that $(B_\pm,\tau_{B_{\pm}})$ is locally compact. Note that $\mathcal{P}^{-1}(\tilde{ \mathcal{W}})$ is an open cover of $\mathcal{P}^{-1}(\overline{\mathcal{P}(\overline{U}_x)}) \supset \overline{U}_x$, and hence so is $\tilde{ \mathcal{V}}$. Therefore, for each $(t, s, \ast )\in \overline{U}_x$, there exists $k\in K$ such that $(t, s, \ast )\in V_k= U_{(t', s',+)}\cup U_{(t',s',-)}$ for some $(t',s')\in B$. If both $U_{(t',s', -)}=\lbrace B_\pm\setminus \overline{U}_x \rbrace= U_{(t',s', +)}$, then $V_k \cap \overline{U}_x=\emptyset$, which contradicts $ (t, s, \ast )\in V_k$. Hence, $(t, s, \ast ) \in U_{(t',s',\ast)}\in \tilde{\mathcal{U}}$ for some $\ast \in \{-,+\}$, and so $\tilde{ \mathcal{U}}$ is an open subcover of $\overline{U}_x$.

We have shown that $B_\pm$ admits a (Hausdorff) one-point compactification, that we denote by $B_\pm \cup\lbrace \tilde{\infty} \rbrace $. We will see that $B_\pm \cup\lbrace \tilde{\infty}\rbrace $ is a sequential space by proving that more generally, it is a first-countable space, i.e., each point of $B_\pm \cup\lbrace \tilde{\infty} \rbrace$ has a countable neighbourhood basis. By definition, the open sets in $B_\pm \cup\lbrace \tilde{\infty} \rbrace $ are all sets that are open in $B_\pm$, together with all sets of the form $(B_\pm \setminus C)\cup \lbrace \tilde{\infty} \rbrace$, where $C$ is any closed and compact set in $B_\pm$. For each $(t,s,\ast)\in B_\pm$, a local basis can be chosen to be the collection of sets $\{U_n\}_{n\in \N}$ given by
$$ U_{n}\defeq \left((t-1/n, t+1/n)\times \left((s-1/n,-), (s+1/n,+) \right)\right)\cap B_\pm.$$
In order to find a local basis for $\tilde{\infty}$, for each $N \in \N$, let
$$ C_{N}\defeq [0,N]\times ((-N, -), (N, +)),$$
and note that by Observation \ref{rem_cyclicIrr}, $C_{N}$ equals its closure. Thus, reasoning as in the first part of the proof of this lemma, one can see that $C_{N}$ is compact, and therefore $ \left\lbrace (B_\pm \setminus C_N) \cup \lbrace \tilde{\infty} \rbrace \right\rbrace_{N \in \N}$ forms a local basis for $\lbrace \tilde{\infty}\rbrace$. Thus, we have shown that $B_\pm \cup\lbrace \tilde{\infty} \rbrace$ is a sequential space.

Finally, if $\lbrace x_k \rbrace_{k\in \N} \subset B\cup\lbrace\tilde{\infty}\rbrace$ is a sequence such that $x_k\to\tilde{\infty}$ as $k\rightarrow \infty$, then for every $N\in \N$, there exists $K(N)\in \N$ such that for all $k \geq K(N)$, $x_k \in (B_\pm \setminus C_N)$. Hence, for all $k\geq K(N)$, $\mathcal{P}(x_k) \in \mathcal{P}(B_\pm \setminus C_N) \subset \R^2 \setminus ([0,N] \times [-N,N])$. Therefore, 
$$ N \to \infty \iff K(N) \to \infty \iff x_k \to \tilde{\infty} \iff \mathcal{P}(x_k) \to \infty,$$ 
and the last claim of the statement follows.
\end{proof}
\section{The semiconjugacy}\label{sec_semiconj}
In this last section we prove a more precise version of Theorem~\ref{thm_main_intro}, namely, Theorem~\ref{thm_main}. We start by bringing together the parameters and functions that will be involved.

\begin{discussion}[Combination of previous results]\label{dis_setting_semi}
Let $f \in \CB$ be an arbitrary but fixed, strongly postcritically separated function. Let us fix a pair of hyperbolic orbifolds $\Or=(S,\nu)$ and $\Ort=(\widetilde{S},\tilde{\nu})$ associated to $f$ provided by Theorem~\ref{thm_orbifolds}. In particular, by Theorem~\ref{thm_orbifolds}\ref{item_a_assocorb}, $S=\C \setminus \overline{U}$, where $\overline{U}$ is a, possibly empty, compact set. By Observation \ref{rem_setting} and since $f\in \B$, we can fix $K>0$ sufficiently large so that
\begin{equation}\label{eq_dk2}
\lbrace P_J \setminus I(f),\overline{U},S(f), 0, f(0) \rbrace \subset \overline{\D_{K/2}} \subset \D_K.
\end{equation}

Let us choose a constant $L>K$ such that $f(\D_K) \Subset \D_L$, and apply Theorem~\ref{thm_CB} to $f$ with this constant. This provides us with a disjoint type map  $g(z)\defeq f(\lambda z)$, for some $\lambda\defeq \lambda(L) \in \C\setminus \{0\}$, and a continuous function $\theta\colon J(g)\to J(f)$, that, in particular, conjugates $g$ and $f$ on subsets of their Julia sets. Moreover, $\theta$ determines an order-preserving bijection between any choice of $\Addr(g)$ and $\Addr(f)$, and the connected components of $\theta(J(g))$ form a valid initial configuration in the sense of Definition \ref{def_initial_conf}, see Observation \ref{obs_CB_conf}. We denote this initial configuration 
\begin{equation}\label{eq_initial_semi}
\left\{\gamma^0_\s\right\}_{\s \in \Addr(f)}, \quad \text{ where }  \gamma^0_{\s} \defeq \theta (J_{\s}) \quad \text{ for each } \s \in \Addr(g).
\end{equation}
Since $f$ is strongly postcritically separated, $f$ satisfies all the assumptions in Theorem~\ref{thm_inverse_hand}, see Observation \ref{rem_setting}. Thus, we can fix the choice of $\Addr(f)$ provided by this theorem, and define the space of signed external addresses $\Addr(f)_\pm$ from $\Addr(f)$, see Definition~\ref{defn_signedaddr}. Using the initial configuration in \eqref{eq_initial_semi}, we define canonical rays $\Gamma(\s, \ast)$ and the set of curves
$$\mathcal{C}\defeq \left\{\gamma^n_{(\s, \ast)} : n\geq 0 \text{ and }(\s, \ast)\in \Addr(f)_\pm\right\}$$ provided by Theorem~\ref{thm_signed}. In particular, by Theorem~\ref{thm_inverse_hand}, for each curve $\gamma^n_{(\s, \ast)} \in \mathcal{C}$, there exists a neighbourhood $\tau_n(\s, \ast) \supset \gamma^n_{(\s, \ast)}$ where we can define an inverse branch of $f$
\begin{equation}\label{eq_inverse_semi}
 f^{-1,[n]}_{(\s,\ast)}\defeq \left(f\vert_{ \tau_n(\s, \ast)}\right)^{-1} \colon f(\tau_n(\s, \ast)) \rightarrow \tau_n(\s, \ast),
\end{equation}
as well as an inverse branch of $f^n$ provided by Observation \ref{obs_chain_inverse},
\begin{equation}\label{eq_inverse_semi2}
f^{-n}_{(\s,\ast)}\defeq \left(f^n\vert_{ \tau_n(\s, \ast)}\right)^{-1} \colon f^n(\tau_n(\s, \ast)) \rightarrow \tau_n(\s, \ast),
\end{equation}
having both of them properties we shall use later on. Next, using the function $g$, we fix a model space $(J(g)_\pm, \tau_J)$ for $f$ (see Definition \ref{defn_modelspace}) and the corresponding associated model function 
$$\tilde{g}\colon J(g)_\pm\rightarrow J(g)_\pm; \quad  \quad \tilde{g}(z, \ast)\mapsto(g(z),\ast).$$ In addition, 
$$\pi\colon J(g)_{\pm} \rightarrow J(g); \quad \quad \pi(z, \ast)\mapsto z,$$
is the projection function from Proposition~\ref{prop_contmodel}. Finally, for each $(\s, \ast) \!\in \!\Addr(g)_\pm$, we denote 
\begin{equation}\label{def_Jsmodel}
 J_{(\s, \ast)}\defeq J_\s \times \{\ast\} \quad \text{ and } \quad I_{(\s, \ast)}\defeq I_\s \times \{\ast\},
\end{equation}
where $I_\s=J_\s\cap I(f)$, see Observation \ref{obs_inverse_CB}. In particular, since $g$ is of disjoint type, 
\begin{equation}
J(g)_\pm=\bigcup_{(\s, \ast) \in \Addr(g)_\pm} J_{(\s,\ast)} \quad \text{ and } \quad \tilde{g}^n(J_{(\s, \ast)})=J_{(\sigma^n(\s), \ast)} \text{ for all } n \geq 0, 
\end{equation}
see Observation \ref{obs_addrz}. Moreover, for each $x \in J(g)_\pm$, $\addr(x)$ denotes the unique $(\s, \ast) \!\in \!\Addr(g)_\pm$ such that $x\in J_{(\s, \ast)}$, see Observation \ref{obs_addrz}.
\end{discussion} 

After setting in \ref{dis_setting_semi} the functions we shall use in the proof of Theorem~\ref{thm_main}, for ease of understanding, we now comment on the main ideas of this proof. For the functions $f$ and $\tilde{g}$ fixed in \ref{dis_setting_semi}, we aim to obtain the function $\phi \colon J(g)_\pm \rightarrow J(f)$ that semiconjugates them as a limit of functions $\phi_n\colon J(g)_\pm \rightarrow J(f)$, that are successively \textit{better approximations} of $\phi$. For each $x\in J(g)_\pm$ and each $n\geq 0$, roughly speaking, $\phi_n(x)$ is defined the following way: we iterate $x$ under the model function $\tilde{g}$ a number $n$ of times. In particular, if $\addr(x) = (\s, \ast)$, then $\tilde{g}^n(x)\subset J_{(\sigma^n(\s), \ast)}$. Then, we use the functions $\pi$ and $\theta$ to \textit{move} from the space $J(g)_\pm$ to the dynamical plane of $f$. More precisely, if $\pi(x)=z$, then $\theta(g^n(z)) \in \gamma^0_{(\sigma^n(\s), \ast)}$. Then, we use the composition of $n$ inverse branches of $f$ of the form specified in \eqref{eq_inverse_semi} to obtain, thanks to Theorem~\ref{thm_inverse_hand} and Observation \ref{obs_chain_inverse}, a point in $\gamma^n_{(\s, \ast)}$, that is $\phi_n(x)$; see Figure \ref{figure_defphi_n}.
\begin{figure}[htb]
	\centering
\begingroup%
  \makeatletter%
  \providecommand\color[2][]{%
    \errmessage{(Inkscape) Color is used for the text in Inkscape, but the package 'color.sty' is not loaded}%
    \renewcommand\color[2][]{}%
  }%
  \providecommand\transparent[1]{%
    \errmessage{(Inkscape) Transparency is used (non-zero) for the text in Inkscape, but the package 'transparent.sty' is not loaded}%
    \renewcommand\transparent[1]{}%
  }%
  \providecommand\rotatebox[2]{#2}%
  \newcommand*\fsize{\dimexpr\f@size pt\relax}%
  \newcommand*\lineheight[1]{\fontsize{\fsize}{#1\fsize}\selectfont}%
  \ifx\svgwidth\undefined%
    \setlength{\unitlength}{396.8503937bp}%
    \ifx\svgscale\undefined%
      \relax%
    \else%
      \setlength{\unitlength}{\unitlength * \real{\svgscale}}%
    \fi%
  \else%
    \setlength{\unitlength}{\svgwidth}%
  \fi%
  \global\let\svgwidth\undefined%
  \global\let\svgscale\undefined%
  \makeatother%
  \begin{picture}(1,0.64285714)%
    \lineheight{1}%
    \setlength\tabcolsep{0pt}%
    \put(0,0){\includegraphics[width=\unitlength,page=1]{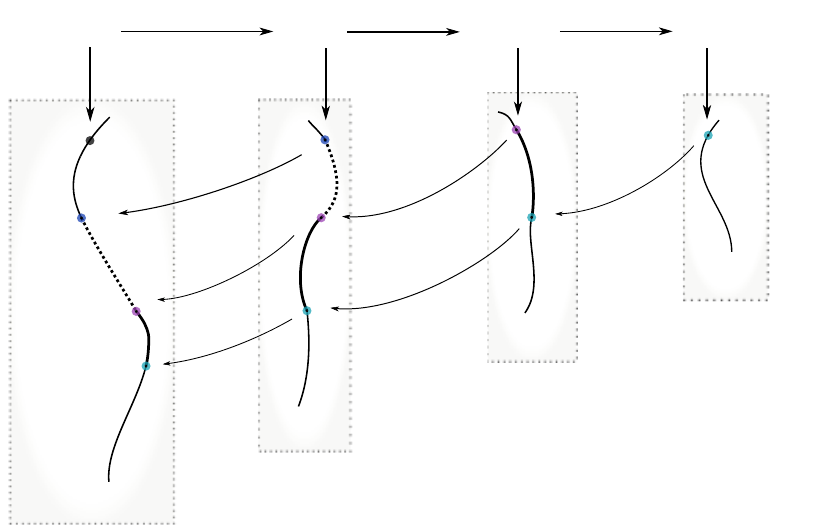}}%
    \put(0.06885787,0.55097077){\color[rgb]{0,0,0}\makebox(0,0)[lt]{\lineheight{1.25}\smash{\begin{tabular}[t]{l}$\varphi_0$\\\end{tabular}}}}%
    \put(0.10365825,0.59811102){\color[rgb]{0,0,0}\makebox(0,0)[lt]{\lineheight{1.25}\smash{\begin{tabular}[t]{l}$x$\end{tabular}}}}%
    \put(0.24948594,0.61641088){\color[rgb]{0,0,0}\makebox(0,0)[lt]{\lineheight{1.25}\smash{\begin{tabular}[t]{l}$\tilde{g}$\end{tabular}}}}%
    \put(0.3476481,0.59800039){\color[rgb]{0,0,0}\makebox(0,0)[lt]{\lineheight{1.25}\smash{\begin{tabular}[t]{l}$\tilde{g}(x)$\end{tabular}}}}%
    \put(0.46950918,0.61812749){\color[rgb]{0,0,0}\makebox(0,0)[lt]{\lineheight{1.25}\smash{\begin{tabular}[t]{l}$\tilde{g}$\end{tabular}}}}%
    \put(0.58302438,0.5979423){\color[rgb]{0,0,0}\makebox(0,0)[lt]{\lineheight{1.25}\smash{\begin{tabular}[t]{l}$\tilde{g}^2(x)$\end{tabular}}}}%
    \put(0.58211806,0.55474563){\color[rgb]{0,0,0}\makebox(0,0)[lt]{\lineheight{1.25}\smash{\begin{tabular}[t]{l}$\varphi_0$\\\end{tabular}}}}%
    \put(0.02707864,0.46642075){\color[rgb]{0,0,0}\makebox(0,0)[lt]{\lineheight{1.25}\smash{\begin{tabular}[t]{l}$\varphi_0(x)$\\\end{tabular}}}}%
    \put(0.01466578,0.37245731){\color[rgb]{0,0,0}\makebox(0,0)[lt]{\lineheight{1.25}\smash{\begin{tabular}[t]{l}$\varphi_1(x)$\\\end{tabular}}}}%
    \put(0.07802527,0.262754){\color[rgb]{0,0,0}\makebox(0,0)[lt]{\lineheight{1.25}\smash{\begin{tabular}[t]{l}$\varphi_2(x)$\\\end{tabular}}}}%
    \put(0.09477709,0.19495234){\color[rgb]{0,0,0}\makebox(0,0)[lt]{\lineheight{1.25}\smash{\begin{tabular}[t]{l}$\varphi_3(x)$\\\end{tabular}}}}%
    \put(0.40205596,0.47180553){\color[rgb]{0,0,0}\makebox(0,0)[lt]{\lineheight{1.25}\smash{\begin{tabular}[t]{l}$\varphi_0(\tilde{g}(x))$\\\end{tabular}}}}%
    \put(0.39013192,0.35351285){\color[rgb]{0,0,0}\makebox(0,0)[lt]{\lineheight{1.25}\smash{\begin{tabular}[t]{l}$\varphi_1(\tilde{g}(x))$\\\end{tabular}}}}%
    \put(0.37657855,0.2393222){\color[rgb]{0,0,0}\makebox(0,0)[lt]{\lineheight{1.25}\smash{\begin{tabular}[t]{l}$\varphi_2(\tilde{g}(x))$\\\end{tabular}}}}%
    \put(0.7309367,0.61927964){\color[rgb]{0,0,0}\makebox(0,0)[lt]{\lineheight{1.25}\smash{\begin{tabular}[t]{l}$\tilde{g}$\end{tabular}}}}%
    \put(0.82721269,0.59827995){\color[rgb]{0,0,0}\makebox(0,0)[lt]{\lineheight{1.25}\smash{\begin{tabular}[t]{l}$\tilde{g}^3(x)$\end{tabular}}}}%
    \put(0.3500939,0.5552378){\color[rgb]{0,0,0}\makebox(0,0)[lt]{\lineheight{1.25}\smash{\begin{tabular}[t]{l}$\varphi_0$\\\end{tabular}}}}%
    \put(0.81890286,0.55608284){\color[rgb]{0,0,0}\makebox(0,0)[lt]{\lineheight{1.25}\smash{\begin{tabular}[t]{l}$\varphi_0$\\\end{tabular}}}}%
    \put(0.63276978,0.48642789){\color[rgb]{0,0,0}\makebox(0,0)[lt]{\lineheight{1.25}\smash{\begin{tabular}[t]{l}$\varphi_0(\tilde{g}^2(x))$\\\end{tabular}}}}%
    \put(0.65488494,0.35706613){\color[rgb]{0,0,0}\makebox(0,0)[lt]{\lineheight{1.25}\smash{\begin{tabular}[t]{l}$\varphi_1(\tilde{g}^2(x))$\\\end{tabular}}}}%
    \put(0.86606505,0.46935196){\color[rgb]{0,0,0}\makebox(0,0)[lt]{\lineheight{1.25}\smash{\begin{tabular}[t]{l}$\varphi_0(\tilde{g}^3(x))$\\\end{tabular}}}}%
    \put(0.09621703,0.02858903){\color[rgb]{0,0,0}\makebox(0,0)[lt]{\lineheight{1.25}\smash{\begin{tabular}[t]{l}$\gamma^3_{(\underline{s}, *)}$\end{tabular}}}}%
    \put(0.32526093,0.11842221){\color[rgb]{0,0,0}\makebox(0,0)[lt]{\lineheight{1.25}\smash{\begin{tabular}[t]{l}$\gamma^2_{(\sigma(\underline{s}), *)}$\end{tabular}}}}%
    \put(0.59959205,0.22581635){\color[rgb]{0,0,0}\makebox(0,0)[lt]{\lineheight{1.25}\smash{\begin{tabular}[t]{l}$\gamma^1_{(\sigma^2(\underline{s}), *)}$\end{tabular}}}}%
    \put(0.83300808,0.3080798){\color[rgb]{0,0,0}\makebox(0,0)[lt]{\lineheight{1.25}\smash{\begin{tabular}[t]{l}$\gamma^0_{(\sigma^3(\underline{s}), *)}$\end{tabular}}}}%
    \put(0.1238784,-0.01879229){\color[rgb]{0.4,0.4,0.4}\makebox(0,0)[lt]{\lineheight{1.25}\smash{\begin{tabular}[t]{l}\fontsize{9pt}{1em}$\tau^3(\underline{s},*)$\end{tabular}}}}%
    \put(0.36134649,0.07061125){\color[rgb]{0.4,0.4,0.4}\makebox(0,0)[lt]{\lineheight{1.25}\smash{\begin{tabular}[t]{l}\fontsize{9pt}{1em}$\tau^2(\sigma(\underline{s}),*)$\end{tabular}}}}%
    \put(0.64123399,0.17912979){\color[rgb]{0.4,0.4,0.4}\makebox(0,0)[lt]{\lineheight{1.25}\smash{\begin{tabular}[t]{l}\fontsize{9pt}{1em}$\tau^1(\sigma^2(\underline{s}),*)$\end{tabular}}}}%
    \put(0.86875572,0.25616289){\color[rgb]{0.4,0.4,0.4}\makebox(0,0)[lt]{\lineheight{1.25}\smash{\begin{tabular}[t]{l}\fontsize{9pt}{1em}$\tau^0(\sigma^3(\underline{s}),*)$\end{tabular}}}}%
    \put(0.21181309,0.03307293){\color[rgb]{0,0,0}\makebox(0,0)[lt]{\begin{minipage}{0.00457828\unitlength}\raggedright \end{minipage}}}%
    \put(0.69632983,0.44380553){\color[rgb]{0,0,0}\makebox(0,0)[lt]{\lineheight{1.25}\smash{\begin{tabular}[t]{l}\fontsize{9pt}{1em} $f^{-1,[1]}_{(\sigma^2(\underline{s}),*)}$\end{tabular}}}}%
    \put(0.2267857,0.4325813){\color[rgb]{0,0,0}\makebox(0,0)[lt]{\lineheight{1.25}\smash{\begin{tabular}[t]{l}\fontsize{9pt}{1em} $f^{-1,[1]}_{(\underline{s},*)}$\end{tabular}}}}%
    \put(0.22235448,0.32219301){\color[rgb]{0,0,0}\makebox(0,0)[lt]{\lineheight{1.25}\smash{\begin{tabular}[t]{l}\fontsize{9pt}{1em} $f^{-1,[2]}_{(\underline{s},*)}$\end{tabular}}}}%
    \put(0.22613423,0.24187309){\color[rgb]{0,0,0}\makebox(0,0)[lt]{\lineheight{1.25}\smash{\begin{tabular}[t]{l}\fontsize{9pt}{1em} $f^{-1,[3]}_{(\underline{s},*)}$\end{tabular}}}}%
    \put(0.45453303,0.43102334){\color[rgb]{0,0,0}\makebox(0,0)[lt]{\lineheight{1.25}\smash{\begin{tabular}[t]{l}\fontsize{9pt}{1em} $f^{-1,[1]}_{(\sigma(\underline{s}),*)}$\end{tabular}}}}%
    \put(0.44739388,0.31767172){\color[rgb]{0,0,0}\makebox(0,0)[lt]{\lineheight{1.25}\smash{\begin{tabular}[t]{l}\fontsize{9pt}{1em} $f^{-1,[2]}_{(\sigma(\underline{s}),*)}$\end{tabular}}}}%
  \end{picture}%
 \endgroup 
	\caption{A schematic of the functions and curves involved in the definition of the maps $\left\{\phi_n\right\}_{n\in \N}$.}
	\label{figure_defphi_n}
\end{figure}

Since all the functions involved in the definition of $\phi_n$ are continuous, continuity of $\phi_n$ will follow. Moreover, we use Theorem~\ref{thm_orbifolds}, that is, orbifold expansion of $f$ on a neighbourhood of $J(f)$, to show that the functions $\phi_n$ converge to a limit function $\phi$ in Lemma~\ref{lem_UHSC}. Finally, using that since $J(g)$ is a Cantor bouquet and $g$ is of disjoint type, all but some of the endpoints of the hairs of $J(g)$ are escaping, Observation \ref{obs_inverse_CB}, we show surjectivity of $\phi$.

\begin{defn}[Functions $\phi_n$]\label{def_phin} Following \ref{dis_setting_semi}, for each $n\geq 0$, we define the function $\phi_n \colon J(g)_{\pm} \to J(f)$ as
$$\phi_0(x)\defeq \theta(\pi(x)) \qquad \text{and}\qquad \phi_{n+1}(x)\defeq f^{-1, [n]}_{\addr(x)}(\phi_{n}(\tilde{g}(x))).$$
\end{defn}

We claim that these functions are well-defined. Indeed, the function $\phi_0$ is well-defined since $\pi(J(g)_\pm)\subset J(g)$. For each $n\geq 1$, let $x\in J(g)_\pm$ and suppose that $\addr(x)=(\s, \ast)$. Then, expanding definitions and using Observation \ref{obs_chain_inverse}, 
\begin{equation}\label{eq_expdef}
\begin{split}
\phi_{n}(x)&=\left(f^{-1, [n]}_{(\s, \ast)}\circ f^{-1, [n-1]}_{(\sigma(\s), \ast)} \circ \cdots \circ f^{-1, [1]}_{(\sigma^{n-1}(\s), \ast)} \circ \phi_{0} \circ \tilde{g}^n\right)(x)\\
&= f^{-n}_{(\s, \ast)} (\phi_0(\tilde{g}^n(x))).
\end{split}
\end{equation}
Since the equalities in \eqref{eq_expdef} only depend on $\addr(x)$ but not on the point $x$ itself, the action of $\phi_{n}$ can be expressed in terms of the sets in \eqref{def_Jsmodel} as
\begin{equation}\label{eq_defphikabbr}
\phi_n\vert_{J_{(\s, \ast)}} \equiv f^{-n}_{(\s, \ast)} \circ \phi_0 \circ \tilde{g}^n \vert_{J_{(\s, \ast)}}.
\end{equation}
Thus, since each $x \in J(g)_\pm$ belongs to a unique set $J_{(\s, \ast)}$ for $(\s, \ast)=\addr(x)$ as $g$ is of disjoint type, see Observation \ref{obs_addrz}, $\phi_n$ is a well-defined function for all $n\geq 0$. In particular, using \eqref{eq_initial_semi} and Observation \ref{obs_chain_inverse}, we have that for each $(\s, \ast)\in \Addr(g)_\pm$,
\begin{equation}\label{eq_phin_gamman}
\begin{split}
\phi_n(J_{(\s,\ast)})&=f^{-n}_{(\s, \ast)}(\phi_0(\tilde{g}^n(J_{(\s, \ast)})))=f^{-n}_{(\s, \ast)}(\phi_0(J_{(\sigma^n(\s), \ast)} ))=f^{-n}_{(\s, \ast)}(\gamma^0_{(\sigma^{n}(\s), \ast)}) \\
&=\gamma^n_{(\s,\ast)}.
\end{split}
\end{equation}

Moreover, by construction, for all $n\geq 0$, 
\begin{equation} \label{commute1}
\phi_{n}\circ \tilde{g}= f \circ \phi_{n+1}.
\end{equation}

\begin{prop}[Continuity of $\phi_n$] \label{new_prop_contphi_k} For each $n\geq 0$, the function $\phi_n \colon J(g)_{\pm} \rightarrow J(f)$ is continuous.
\end{prop}
\begin{proof}
The function $\phi_0$ is continuous because it is the composition of two continuous functions, see Theorem~\ref{thm_CB} and Proposition~\ref{prop_contmodel}. Fix any $n\geq 1$, fix an arbitrary $x\in J(g)_\pm$, let $\addr(x)\eqdef (\s, \ast)$ and let $\Upsilon^n(\s, \ast)$ be the interval in $\Addr(f)_\pm$ provided by Theorem~\ref{thm_inverse_hand}. By Theorem~\ref{thm_CB}\ref{item:addr}, $\theta$ establishes a one-to-one and order-preserving correspondence between $\Addr(f)$ and $\Addr(g)$. Hence, up to this correspondence, the topological spaces $\Addr(f)_\pm$ and $\Addr(g)_\pm$ are the same, and so, $\Upsilon^n(\s, \ast)$ is an open interval in $(\Addr(g)_\pm, \tau_A)$. Let us consider the subset of $J(g)_\pm$
$$ A\defeq \bigcup_{(\ultau, \star)\in \Upsilon^n(\s, \ast)} J_{(\ultau,\star)}.$$ Then, by Observation \ref{obs_chain_inverse} and \eqref{eq_defphikabbr},
\begin{equation} \label{eq_pruebacontphi_n}
\phi_n\vert_{ A} \equiv f^{-n}_{(\s, \ast)} \circ \phi_0 \circ \tilde{g}^n\vert_{ A}.
\end{equation}
It follows that $\phi_n\vert_{A}$ is continuous as it is a composition of continuous functions: we have just shown that $\phi_0$ is continuous, and $\tilde{g}$ is continuous by Proposition~\ref{prop_contmodel}. Moreover, by Observation \ref{obs_chain_inverse}, it holds that $\phi_0(\tilde{g}^n(A)) \subset f^n(\tau_n(\s, \ast))$, and thus, the restriction of $f^{-n}_{(\s, \ast)}$ to $\phi_0 (\tilde{g}^n (A))$ is well-defined and continuous.

We are only left to show that $A$ contains an open neighbourhood of $x$. Recall that we defined in Proposition~\ref{prop_mapC} an open map $\mathcal{C}:(\Addr(g)_\pm, \tau_A) \rightarrow (\R \setminus \Q\times \{-,+ \}, \tau_I)$. 
In particular, $\mathcal{C}(\Upsilon^n(\s, \ast))$ is an open interval in $(\R\setminus \Q \times \{ -,+ \},\tau_I)$. Let $\tilde{\psi}\colon J(g)_{\pm} \to B_\pm$ be the bijection from \ref{dis_topology_model}.
In particular, $\tilde{\psi}(x)=(t, \mathcal{C}(\s, \ast))$ for some $t>0$. Then, $U \defeq ((t_1, t_2) \times \mathcal{C}(\Upsilon^n(\s, \ast))) \cap B_\pm$ is an open neighbourhood of $\tilde{\psi}(x)$ in $B_\pm$ for any choice of $t_1,t_2\in \R^+$ such that $t_1 <t<t_2$. Consequently, see \ref{dis_topology_model}, $\tilde{\psi}^{-1}(U)$ is an open neighbourhood of $x$ that lies in $A$.
\end{proof}

\subsection*{Convergence to the semiconjugacy}

For the hyperbolic orbifold $\Or\supset J(f)$ fixed in \ref{dis_setting_semi}, let $d_\Or$ be the distance function defined from its orbifold metric; see \ref{orb_metric}. We shall now see that for any given point $x\in J(g)_\pm$, $d_\Or(\phi_{n+1}(x),\phi_n(x))\rightarrow 0$ as $n\rightarrow \infty$.
\begin{lemma}[The functions $\phi_n$ form a Cauchy sequence] \label{lem_UHSC} There exists constants $\mu, \Lambda>1$ such that for each $x \in J(g)_\pm$, 
\begin{enumerate}[label=(\Alph*)]
\item \label{itemA_cauchy} $d_\Or(\pi(x),\phi_{0}(x))<\mu$,
\item \label{itemB_cauchy} $d_\Or(\phi_{n+1}(x),\phi_n(x))\leq \frac{\mu}{\Lambda^{n}}$ for every $n \geq 0$.
\end{enumerate}
\end{lemma}
\begin{proof}
Fix $x\in J(g)_\pm$ and let $z\defeq \pi(x)$. In particular, $\phi_0(x)=\theta(z)$. By \eqref{eq_dk2} and Theorem~\ref{thm_CB}\ref{item:inclusions}\ref{item:annulus},
\begin{equation}\label{eq_thetaJk}
\theta(J(g))\cup J(g) \subset \C \setminus \D_{K} \subset \C \setminus \D_{K/2} \subset \Or \quad \text{ and } \quad \theta(z) \in \overline{A(M^{-1} \vert z\vert, M\vert z\vert)},
\end{equation}
for some $M>1$ that does not depend on the point $z$.

By Theorem~\ref{thm_orbifolds}\ref{lem_annulus}, there exists $\tilde{R}>1$ so that if $\overline{A(r, 2r)}\subset A(r/2,4 r) \subset \Or$, then the $\Or$-distance between any two points in $\overline{A(r, 2r)}$ is less than $\tilde{R}$. 
We want to combine this result with \eqref{eq_thetaJk} to get an upper bound for $d_\Or(z,\theta(z))$ by expressing the annulus in \eqref{eq_thetaJk} as a finite union of annuli of the form $A(r, 2r)$ for some $r>0$. More specifically, let $N$ be the smallest natural number for which $2^N\geq M^{2}$. That is, $N\defeq \left \lceil{\frac{2\log M}{\log 2}}\right \rceil $, and let $r\defeq \max\{K, M^{-1}\vert z \vert\}$. Then, by \eqref{eq_thetaJk},
\begin{equation}\label{eq_annuli}
z,\: \theta(z) \in \bigcup_{i=1}^{N} \overline{A\left(2^{i-1}r, 2^{i}r \right)} \subset \bigcup_{i=1}^{N} A\left(2^{i-2}r, 2^{i+1}r \right) \subset\Or.
\end{equation}
Thus, since the constant $N$ does not depend on the point $z\in J(g)$, we have that for all $x\in J(g)_\pm$, $d_\Or(\phi_0(x),\pi(x))\leq N \cdot \tilde{R} \eqdef\mu_1$, and item \ref{itemA_cauchy} is proved. In particular, $\mu_1>1$.

We now prove item \ref{itemB_cauchy}. Let $\addr(x)=(\s, \ast)$ and fix any $n\in \N$. Recall that by \eqref{eq_phin_gamman}, since by Theorem~\ref{thm_signed}\ref{item:tails_bijection}, $\gamma^n_{(\s, \ast)} \subseteq \gamma^{n+1}_{(\s, \ast)}$,
$$\phi_n(x), \phi_{n+1}(x) \in \gamma^{n+1}_{(\s, \ast)}.$$
Thus, the $\Or$-length of the piece of $\gamma^{n+1}_{(\s, \ast)}$ that joins $\phi_n(x)$ and $\phi_{n+1}(x)$ provides an upper bound for the $\Or$-distance between these two points. Let $\delta(n)$ be that curve. Then, using \eqref{eq_defphikabbr} and \eqref{commute1}, we have $$f^n(\phi_n(x))=\phi_0(\tilde{g}^n(x)), \qquad f^n(\phi_{n+1}(x))=\phi_1(\tilde{g}^n(x)),$$ and
$\delta(1)\defeq f^n(\delta(n)) \subset \gamma^1_{(\sigma^n(\s), \ast)}$ is a curve with endpoints $\phi_0(\tilde{g}^n(x))$ and $\phi_1(\tilde{g}^n(x))$ and such that $f^{-n}_{(\s, \ast)}(\delta(1))=\delta(n)$; see Figure \ref{figure_defphi_n}. Since $f \in\B$ and is strongly postcritically separated, by Corollary~\ref{cor_uniform}, any upper bound for $\ell_\Or(\delta(1))$ is also an upper bound for $\ell_\Or(\delta(n))$. In particular, if we find a constant $C$ that bounds the $\Or$-length of the sub-curve in $\gamma^1_{\addr(y)}$ between $\phi_0(y)$ and $\phi_1(y)$ for all $y\in J(g)_\pm$, being $C$ independent of the point $y$, then \ref{itemB_cauchy} would follow. However, those sub-curves are pieces of ray tails, and in principle might not be rectifiable. Therefore and instead, we find curves in their post-$0$-homotopy class (see Definition \ref{def_post0}) with bounded orbifold length. More specifically:

\begin{Claim}
There exists a constant $\mu_2>0$ such that for each $x\in J(g)_\pm$ and $n\geq 0$, if $\delta(1)$ is the piece of $\gamma^1_{\addr(x)}$ joining $\phi_0(x)$ and $\phi_1(x)$, then there exists $\tilde{\delta}(1) \in [\delta(1)]_{_0}$ with $\ell_{\Or}(\tilde{\delta}(1))\leq \mu_2.$
\end{Claim}

\begin{subproof}
Fix an arbitrary $x\in J(g)_\pm$ and let $z\defeq \pi(x)$. If $\phi_{0}(x)=\phi_{1}(x)$, the claim holds trivially. Otherwise, since $f\vert_{\gamma_{\addr(x)}^1}$ is injective,
\begin{equation}
\phi_{0}(x) \neq \phi_{1}(x) \iff f(\phi_0(x))=f(\theta(z))\neq \theta(g(z))=f(\phi_1(x)).
\end{equation}
But by Theorem~\ref{thm_CB}\ref{item:boundedC}, we have that $f(\theta(z))\neq \theta(g(z))$ only when the piece of ray $\xi$ joining $f(\phi_1(x))$ and $(\phi_0(x))$ belongs to a compact set $C\subset \C\setminus \D_L \cap f^{-1}(\C\setminus \D_L) \subset \Or$.

Using Theorem~\ref{cor_homot2}, we are aiming to find curves post-$0$-homotopic to the pieces of dynamic rays totally contained in $C$ with uniformly bounded $\Or$-length. Observe that $C \cap P(f) \subset I(f)$, since by the choice of the constant $K$ in \eqref{eq_dk2}, $(P(f)\setminus I(f)) \Subset \D_K$. Moreover, by discreteness of $P_J$, $C \cap P_J$ is a finite set, and since, by Theorem~\ref{thm_CB}, $f$ is criniferous, there exists at least one ray tail connecting each point of $C \cap P_J$ to infinity; see \cite[Theorem 1.4]{mio_newCB}. Since $S(f)\subset \D_L$, the connected components of $f^{-1}(\C\setminus \D_L)$ form a collection $\T$ of tracts. Then, see \cite[Lemma~2.1]{lasse_dreadlocks}, at most finitely many pieces of the tracts in $\T$ intersect $C$, say $\{T_1, \ldots, T_m\}$. Each of these pieces is simply-connected and its boundary is an analytic curve, and hence locally connected. Thus, we can apply Theorem~\ref{cor_homot2} to the closure of each $T_i$ in $C$, to obtain a constant $L_i$ such that for any (connected) piece of ray tail $\xi \subset \overline{T_i}\cap C$, there exists $\delta\in [\xi]_{_0}$ with $\ell_{\Or}(\delta)\leq L_i.$ The claim now follows letting $\mu_2\defeq \max_{1\leq i\leq m}L_i$ and using Corollary~\ref{cor_uniform}. 
\end{subproof}	

The claim implies that for each $x\in J(g)_\pm$ and $n\geq 0$, if $\delta(1)$ is the piece of $\gamma^1_{(\s, \ast)}$ joining $\phi_0(\tilde{g}^n(x))$ and $\phi_1(\tilde{g}^n(x))$, then there exists $\tilde{\delta}(1) \in [\delta(1)]_{_0}$ with $\ell_{\Or}(\tilde{\delta}(1))\leq \mu_2.$ Hence, by Proposition~\ref{cor_homot}, if $\delta(n)\subset f^{-n}_{(\s, \ast)}(\delta(1))$ is the curve joining $\phi_0(x)$ and $\phi_1(x)$, then there exists a unique curve $\tilde{\delta}(n) \subseteq f^{-n}_{(\s, \ast)}(\tilde{\delta}(1))$ satisfying $\tilde{\delta}(n) \in [\delta(n)]_{_0}$. In particular, $\tilde{\delta}(n)$ has endpoints $\phi_0(x)$ and $\phi_1(x)$, and moreover, by Corollary~\ref{cor_uniform}, there exists a constant $\Lambda>1$, that does not depend on $x$, such that
$$d_\Or(\phi_{n+1}(x),\phi_n(x))\leq \ell_\Or(\tilde{\delta}(n)) \leq \frac{\ell_\Or(\tilde{\delta}(1))}{\Lambda^{n}} \leq \frac{\mu_2}{\Lambda^{n}}.$$ 
Letting $\mu\defeq \max\{\mu_1, \mu_2\}$, the lemma follows.
\end{proof}

\noindent Finally, we state and prove a more detailed version of Theorem~\ref{thm_main_intro}. Recall from Definition and Theorem \ref{thm_signed} that for each $(\s, \ast)\in \Addr(f)_\pm$ we can define a canonical ray $\Gamma(\s, \ast)$.
\begin{thm}\label{thm_main} Let $f\in \CB$ be strongly postcritically separated, let $J(g)_\pm$ be a model space for $f$ and let $\tilde{g}$ be its associated model function. Then, there exists a continuous surjective function
\begin{equation*}
\phi \colon J(g)_{\pm} \rightarrow J(f) \qquad \text{ so that } \qquad f\circ\phi = \phi\circ \tilde{g},
\end{equation*}
$\phi(I(g)_{\pm})=I(f)$ and there is $K\in \N$ such that for every $z\in I(f)$, $\# \phi^{-1}(z)< K$. 

Moreover, for each $(\s, \ast) \in \Addr(g)_\pm$, the restriction map $\phi\colon J_{(\s, \ast)}\rightarrow\overline{\Gamma(\s, \ast)}$ is a bijection, and so $\overline{\Gamma(\s, \ast)}$ is a canonical ray together with its endpoint.
\end{thm}

\begin{observation}\label{obs_orders_semi} We have implicitly stated in Theorem~\ref{thm_main} that $\phi$ establishes a one-to-one correspondence between $\Addr(g)_\pm$ and $\Addr(f)_\pm$, since with some abuse of notation, we have stated that for each $(\s, \ast)\in \Addr(g)_\pm$, $J_{(\s, \ast)}\subset J(g)_\pm$ is mapped to $\overline{\Gamma(\s, \ast)} \subset J(f)$ for $(\s, \ast) \in \Addr(f)_\pm$. Here, $\overline{\Gamma(\s, \ast)}$ denotes the closure of $\Gamma(\s, \ast)$ in $\C$. In particular, we are claiming that $\phi$ is an order-preserving continuous map.
\end{observation}	

\begin{proof}[Proof of Theorem~\ref{thm_main}]
Since, by Observation \ref{obs_unique_model}, any two models for $f$ are conjugate, we may assume without loss of generality that $g$ is the disjoint type function fixed in \ref{dis_setting_semi}. Then, $J(g)_\pm$ and $\tilde{g}$ are also fixed in \ref{dis_setting_semi}. Let $\{\phi_n\}_{n\geq 0}$ be the sequence of functions given by Definition \ref{def_phin} following \ref{dis_setting_semi}. By Proposition~\ref{new_prop_contphi_k} and Lemma~\ref{lem_UHSC}, $\{\phi_n\}_{n\geq 0}$ is a uniformly Cauchy sequence of continuous functions. Since the orbifold metric in $\Or$ is complete, they converge uniformly to a continuous limit function
$\phi :J(g)_{\pm} \rightarrow\Or,$ which by the functional equation (\ref{commute1}) satisfies
\begin{equation} \label{commutative}
\phi\circ \tilde{g}= f \circ \phi.
\end{equation}
\noindent By Lemma~\ref{lem_UHSC}, there are constants $\mu, \Lambda>1$ such that
\begin{equation*}\label{eq_sum}
\begin{split}
d_\Or(\phi(x), \pi(x)) &\leq d_\Or(\phi(x),\phi_0(x))+ d_\Or(\phi_0(x), \pi(x)) \\ & \leq \sum_{k=0} ^{\infty} d_\Or(\phi_{k+1}(x),\phi_k(x)) +\mu 
\leq 2\sum_{j=0}^{\infty}\frac{\mu}{\Lambda^{j}}=\frac{2\mu\Lambda}{\Lambda-1}.
\end{split}
\end{equation*}

\noindent This means for sequences $\lbrace x_n \rbrace_{n\in \N} \subset J(g)_{\pm}$ that as $n\rightarrow \infty$,
\begin{equation}\label{corrId}
\phi(x_n)\rightarrow\infty \quad \text{ if and only if } \quad \pi(x_n)\rightarrow\infty.
\end{equation}
In particular, this holds when $\lbrace x_n \rbrace_{n\in \N}=\lbrace \tilde{g}^{n}(x) \rbrace_{n\in \N}$ is the orbit of some $x \in I(g)_\pm$. Using that, by \eqref{commutative}, $\phi(\tilde{g}^{n}(x))=f^{n}(\phi(x))$, we have that $x\in I(g)_{\pm}$ if and only if $\phi(x)\in I(f)$. Equivalently, 
\begin{equation}\label{eq_quasisurjective}
\phi(I(g)_\pm) \subseteq I(f) \quad \text{ and } \quad \phi(J(f)_\pm \setminus I(g)_\pm) \subseteq J(f)\setminus I(f).
\end{equation}

Recall from Observation \ref{obs_inverse_CB} that since $g$ is a disjoint type function whose Julia set is a Cantor bouquet, each of the sets $J_\s$ from \eqref{eq_Js} is a dynamic ray together with its endpoint, and hence, contains at most one non-escaping point, namely, its endpoint $e_\s$. Thus, for each $\ast \in \{-,+\}$, 
\begin{equation}\label{eq_e_s}
J_{(\s, \ast)} \setminus I_{(\s, \ast)} \subseteq \{(e_\s, \ast)\}.
\end{equation}

\begin{Claim}\label{claim_main}
For each $(\s,\ast)\in \Addr(g)_\pm$, $\phi\colon I_{(\s, \ast)}\rightarrow \Gamma(\s, \ast)\cap I(f)$ is a bijection.
\end{Claim}
\begin{subproof}
First, recall that by Theorem \ref{thm_CB}, $J(g)$ contains a Cantor bouquet $X$ such that $I(g)\subset \Orb^-(X)$, and so, by Observation \ref{obs_betas}, we have that the set $I_\s$ can be written as a nested sequence of ray tails $\beta^n_\s$, that are $n$-th preimages of the hairs $X_{\sigma^n(\s)}$. In particular, for each $n\in \N$, 
\begin{equation} \label{eq_new1}
g^n(\beta^n_\s)=X_{\sigma^n(\s)}\subset X
\end{equation}
and we can write
\begin{equation*}
\beta^n_{(\s, \ast)}\defeq \beta^n_\s\times \{\ast\} \quad \text{ and } \quad I_{(\s,\ast)}=\bigcup_{n\geq 0} \beta^n_{(\s, \ast)}.
\end{equation*}
Moreover, for each $n\in \N$, let $\hat{\gamma}^n_{(\s, \ast)}\defeq\gamma^n_{(\s, \ast)}\cap I(f)$. Then, by Theorem~\ref{thm_signed}, $\gamma^n_{(\s, \ast)}\setminus \hat{\gamma}^n_{(\s, \ast)}$ is at most the endpoint of $\gamma^n_{(\s, \ast)}$ and \[ \Gamma(\s,\ast)\cap I(f)=\bigcup_{n\geq 0}\hat{\gamma}^n_{(\s,\ast)}.\]

We claim that for each $n\in \N$, $\phi\vert_{\beta^n_{(\s, \ast)}} \equiv \phi_n\vert_{\beta^n_{(\s, \ast)}}$ is a bijection to $\hat{\gamma}^n_{(\s, \ast)}$. Indeed, for all $m\geq n$, by \eqref{eq_expdef}, \eqref{eq_new1} together with Theorem~\ref{thm_CB}\ref{item:commute} and \eqref{eq_initial_semi},
\begin{equation*}
\begin{split}
\phi_{m}(\beta^n_{(\s, \ast)}) &= f^{-m}_{(\s, \ast)}\circ \theta \circ \pi\circ \tilde{g}^{m} (\beta^n_{(\s, \ast)})= f^{-n}_{(\s, \ast)}\circ f^{n-m}\circ \theta \circ g^{m-n} (g^n(\beta^n_{\s}))\\
&=f^{-n}_{(\s, \ast)}\circ\theta\circ g^{n}(\beta^n_{\s})=f^{-n}_{(\s, \ast)}(\hat{\gamma}^0_{(\sigma^n(\s),\ast)})=\hat{\gamma}^n_{(\s, \ast)}.
\end{split}
\end{equation*}
Since this holds for all $m\geq n$, we have shown that 
$$\phi\vert_{\beta^n_{(\s, \ast)}}\equiv \phi_n\vert_{\beta^n_{(\s, \ast)}}=f^{-n}_{(\s, \ast)}\circ\theta\circ g^{n}\vert_{\beta^n_{\s}},$$ which by Theorem~\ref{thm_CB}\ref{item:homeomX} and Observations \ref{obs_betas} and \ref{obs_chain_inverse}, is a composition of bijections.
\end{subproof}
The claim implies that $\phi(J(g)_\pm)= \bigcup_{(\s, \ast)\in \Addr(f)_\pm}\Gamma(\s, \ast) \cap I(f)\supset I(f)$, see \eqref{eq_thmcanonical}. This together with \eqref{eq_quasisurjective} leads to $\phi(I(g)_\pm)=I(f)$. In addition, for each $(\s, \ast)\in \Addr(g)_\pm$, if $(e_\s, \ast)\in J(g)_\pm \setminus I(g)_\pm$, then by \eqref{eq_quasisurjective} and \eqref{eq_e_s}, $\phi(e, \ast)\in J(f) \setminus I(f)$, and so by the previous claim and continuity of $\phi$, $\phi\vert_{ J_{(\s, \ast)}}$ is injective and 
\begin{equation}\label{eq_phiJinGamma}
\Gamma(\s, \ast)\subset \phi(J_{(\s,\ast)})\subset \overline{\Gamma(\s, \ast)}.
\end{equation}
In order to prove surjectivity of $\phi$, let $J(g)_{\pm}\cup\lbrace\tilde{\infty} \rbrace$ be the one point compactification of $J(g)_{\pm}$ provided by Lemma~\ref{compactification}, and denote by $J(f) \cup \lbrace \infty \rbrace$ the compactification of $J(f)$ as a subset of the Riemann sphere $\widehat{\C}$. By Lemma~\ref{compactification} and \eqref{corrId}, given a sequence $\lbrace x_n \rbrace_{n\in \N} \subset J(g)_{\pm}\cup\lbrace\tilde{\infty}\rbrace$, we have 
\begin{equation}\label{eq_iff}
 \lim_{n \rightarrow \infty} x_n=\tilde{\infty} \quad \Longleftrightarrow \quad \lim_{n \rightarrow \infty}\pi(x_n)=\infty \quad \Longleftrightarrow \quad \lim_{n \rightarrow \infty}\phi(x_n)=\infty. 
\end{equation}

Since by Lemma~\ref{compactification} $J(g)_{\pm}\cup\lbrace\tilde{\infty} \rbrace$ is a sequential space, and so is $\widehat{\C}$, the notions of continuity and sequential continuity for functions between these spaces are equivalent. Therefore, by \eqref{eq_iff}, we can extend $\phi$ to a continuous map $\hat{\phi}: J(g)_{\pm} \cup \lbrace \tilde{\infty} \rbrace \rightarrow J(f) \cup \lbrace \infty \rbrace$ by defining $\hat{\phi}(\tilde{\infty})=\infty$. By continuity of $\hat{\phi}$, we have that $\hat{\phi}\left(J(g)_{\pm}\cup\lbrace\tilde{\infty} \rbrace\right)$ is compact. By definition of $\hat{\phi}$, it must be the case that $\hat{\phi}(J(g)_{\pm})=\phi(J(g)_{\pm})$, and by removing $\lbrace\infty\rbrace$ from the codomain of $\hat{\phi}$, we can conclude that $\phi(J(g)_{\pm})$ is (relatively) closed in $J(f)$ with respect to the original topologies. By this and since the Julia set is the closure of the escaping set for any function in $\B$, \cite{eremenkoclassB}, we have
\begin{equation*}
I(f) =\phi(I(g)_\pm) \subset \phi (J(g)_\pm) \subset J(f)= \overline{I(f)}.
\end{equation*}
Consequently, $\phi (J(g)_\pm)$ must be equal to $J(f)$, showing that $\phi$ is surjective. Moreover, arguing exactly the same way, we can see that for each $(\s, \ast)\in \Addr(f)_\pm$, the set $\phi(J_{(\s, \ast)})$ is closed in $J(f)$, and hence, by \eqref{eq_phiJinGamma}, $\phi\colon J_{(\s, \ast)}\rightarrow\overline{\Gamma(\s, \ast)}$ is a bijection. In particular, $\overline{\Gamma(\s, \ast)}$ is a canonical ray together with its endpoint.

Finally, each $z\in I(f)$ belongs to $\#\Addr(z)_\pm=\prod^{\infty}_{j=0}\deg(f,f^{j}(z))$ canonical rays; see Definition \ref{def_addrpm}. By the claim in this proof, for each $z\in I(f)$, $\#\phi^{-1}(z)=\#\Addr(z)_\pm$. Moreover, since $f$ is strongly postcritically separated, by items \ref{itema_defsps} and \ref{itemb_defsps} in Definition~\ref{def_strongps}, there exist constants $N,c\in \N$ such that for each $z\in J(f)$, $\#(\Orb^+(z)\cap \Crit(f)) \leq c$ and $\deg(f,w)\leq N$ for all $w\in \Crit(f)$. Hence, letting $K\defeq N^c$, the claim in the statement follows.
\end{proof}

\begin{proof}[Proof of Theorem~\ref{thm_main_intro}]
It is a direct consequence of Theorem~\ref{thm_main}.
\end{proof}	

\begin{proof}[Proof of Corollary~\ref{cor_intro}]
Note that $f\in \CB$ is in particular criniferous, see Theorem~\ref{thm_CB}, and since it is also strongly postcritically separated, it has no asymptotic values in its Julia set. Hence, by Theorem~\ref{thm_signed}, proving that all canonical rays of $f$ land suffices to conclude that all its dynamic rays land. Since by Theorem~\ref{thm_main}, for each $(\s,\ast)\in \Addr(f)_\pm$, $ \overline{\Gamma(\s,\ast)}$ is a canonical ray together with its landing point, the corollary follows.
\end{proof}

\begin{proof}[Proofs of Theorems \ref{thm_1intro} and \ref{thm_main_intro1}] If $f$ is a finite composition of functions of finite order in $\B$, then $f\in \CB$, see \cite[Proposition~6]{mio_newCB}. Moreover, if $S(f)$ is a finite collection of critical values in $I(f)$, $f$ has bounded criticality on $J(f)$, and $\vert w-z\vert \geq \epsilon\max\{\vert z \vert, \vert w \vert\}$ for some $\epsilon>0$ and all distinct $z,w \in P(f)$, then $f$ is strongly postcritically separated, see Definition \ref{def_strongps}. Thus, Theorems \ref{thm_1intro} and \ref{thm_main_intro1} respectively follow from Corollary~\ref{cor_intro} and Theorem~\ref{thm_main_intro}.
\end{proof}
\bibliographystyle{alpha}
\bibliography{biblioComplex}
\end{document}